\documentclass[sn-mathphys-num]{sn-jnl}

\usepackage{graphicx}%
\usepackage{multirow}%
\usepackage{amsmath,amssymb,amsfonts}%
\usepackage{amsthm}%
\usepackage{mathrsfs}%
\usepackage[title]{appendix}%
\usepackage{xcolor}%
\usepackage{textcomp}%
\usepackage{manyfoot}%
\usepackage{booktabs}%
\usepackage{algorithm}%
\usepackage{algorithmicx}%
\usepackage{algpseudocode}%
\usepackage{listings}%
\usepackage{pifont}
\newcommand{\red}[1]{\textcolor{black}{#1}}

\theoremstyle{thmstyleone}%
\newtheorem{theorem}{Theorem}
\newtheorem{proposition}[theorem]{Proposition}%

\theoremstyle{thmstyletwo}%
\newtheorem{example}{Example}%
\newtheorem{remark}{Remark}%

\theoremstyle{thmstylethree}%
\newtheorem{definition}{Definition}%

\newtheorem{corollary}[theorem]{Corollary}

\begin{document}

\title[]{A polynomial bosonic form of statistical configuration sums and
the odd/even minimal excludant in integer partitions}


\author{\fnm{Taichiro} \sur{Takagi}}

\affil{\orgdiv{Department of Applied Physics}, \orgname{National Defense Academy}, \orgaddress{\state{Kanagawa 239-8686}, \country{Japan}}}

\abstract{Inspired by the study of the minimal excludant in integer partitions
by G.E. Andrews and D. Newman, we introduce a pair of new partition statistics,
sqrank and rerank.
They are related to 
a polynomial bosonic form of statistical configuration sums
for an integrable cellular automaton.
For all nonnegative integers $n$, we prove that
the partitions of $n$ on which sqrank or rerank takes on a particular value, say $r$, are equinumerous with the partitions of $n$ on which the
odd/even minimal exclutant takes on the corresponding value, $2r+1$ or $2r+2$. 
}

\keywords{Gaussian polynomials, Ferrers diagrams, rim hooks, bit sequences, affine crystals.}


\pacs[MSC Classification]{05A17, 05A30, 11P83, 05E10, 82B23.}

\maketitle

\section{Introduction}\label{sec1}
A {\em partition} of a positive integer $n$ is a finite weakly decreasing sequence of
positive integers $\lambda = (\lambda_1, \dots, \lambda_r)$ such that $\sum_{i=1}^r \lambda_i =n$.
It is sometimes written as $\lambda \vdash n$, and each $\lambda_i$ is 
called a {\em part} of $\lambda$.
The \red{set of} integer partitions \red{is} denoted by $\mathscr{P}$, in which
the empty sequence for the only partition of zero (denoted by $\emptyset$) is also 
included as an element.

Recently, Andrews and Newman initiated \red{the study of} the {\em minimal excludant} in
integer partitions and found many remarkable relations between 
different partition statistics \cite{AN2020}.
They introduced the {minimal excludant} function 
${\rm mex}_{A,a} : \mathscr{P} \rightarrow \mathbb{Z}_{>0}$ 
{for integers $0 \leq a < A$}
by letting
${\rm mex}_{A,a}(\lambda)$ be the smallest positive
integer congruent to $a$ modulo $A$ that is not a part of $\lambda$.
They also defined
$p_{A,a}(n)$ to be the number of the partitions of $n$ {for} which
${\rm mex}_{A,a}$ takes on values congruent to $a$ modulo $2A$.
This in particular implies that $p_{2,2}(n)$ is the number of the partitions of $n$
whose `even minimal excludant' is not a multiple of four.
They asked, at least in the author's interpretation,
a question about the existence of
a different partition statistic that can 
reproduce the same integer sequence $\{ p_{2,2}(n) \}_{n \geq 0}$ \cite[p.~9, Questions II]{AN2020}.
This question naturally arises because
they {already} found theorems
on the existence of
such partition statistics that are different from ${\rm mex}_{1,1}$ and ${\rm mex}_{3,3}$ but
can reproduce the same integer sequences $\{ p_{1,1}(n) \}_{n \geq 0}$ and 
$\{ p_{3,3}(n) \}_{n \geq 0}$ \cite[Theorems 2 and 3]{AN2020}.
More specifically, they found that
the relevant partition statistics are Dyson's crank and rank \cite{Dyson1944, AG1988}.

Motivated by {the above} question,
in this paper
we introduce a pair of new partition \red{statistics, which} 
we call {\em sqrank} and {\em rerank}.
For all nonnegative integers $n$, we prove that
the partitions of $n$ on which sqrank or rerank takes on a particular value, say $r$, are equinumerous with the partitions of $n$ on which the
odd/even minimal exclutant takes on the corresponding value, $2r+1$ or $2r+2$. 
This is the main result of this paper (Theorem \ref{th:main}).
As a corollary of this theorem, the above question by Andrews and Newman is also answered.

To establish the main result, we use the fact that
our new partition statistics are associated with
a polynomial bosonic form of statistical configuration sums
for an integrable cellular automaton \cite{KOTY, Takagi1}.
Let us briefly explain the relation between this bosonic form and
a formula by Andrews and Newman \cite{AN2020}, and 
\red{comment on some relevant historical links to} mathematical physics.
We use standard notation
$(z;q)_n = \prod_{k=1}^n (1-z q^{k-1})$ for $n\geq1$, and 
$(z;q)_0=1$.
The Gaussian polynomial ${n \brack m}$ is defined by
\begin{equation}\label{eq:gaussianpoly}
{n \brack m} =
\begin{cases}
(q;q)_n (q;q)_m^{-1}(q;q)_{n-m}^{-1} & \mbox{if } 0 \leq m \leq n, \\
0 & \mbox{otherwise}.
\end{cases}
\end{equation}
As a suggestion to solve the above open question, Andrews and Newman presented
the following expression for the generating function of $p_{2,2}(n)$
\begin{equation*}
\sum_{n \geq 0} p_{2,2}(n) q^n 
=  \sum_{n \geq 0} \frac{q^{n^2+n} B_n(q)}{(q;q)_{2n+1}},
\end{equation*}
where $B_n(q)$ is a polynomial in $q$ that is 
written as an alternating sum of the Gaussian polynomials 
(See \eqref{eq:2024july2_1} in the main text).
Three conjectures about this polynomial were given, {two of which will be presented here} \cite[p.~9, Questions II (1) and II (3)]{AN2020}:
\begin{itemize}
\item $B_n(q)$ has nonnegative coefficients.
\item $B_n(q)$ enumerates some subset of the partitions into at most $n$ parts each $\leq n+1$. 
\end{itemize}
S.~Chern \cite{Chern} gave a proof of the first assertion.
As was mentioned therein, to verify this assertion it suffices to prove that the polynomial
\begin{equation*} 
X^+_{L,s}(q)  = { L \brack s }- { L \brack s-1},
\end{equation*}
has nonnegative coefficients for $1 \leq s \leq L/2$.
However, it was known that this polynomial takes the following 
statistical configuration {sum form}
\begin{equation*} 
X^+_{L,s}(q) =
\sum_{\eta \in \mathcal{H}^{+} (L,s)} q^{E(\eta)},
\end{equation*}
where $\mathcal{H}^{+} (L,s)$ is a subset of the bit sequences
of $s$ ones and $L-s$ zeros, and $E(\eta)$ is an {\em energy} attributed to
bit sequences $\eta$ (See \eqref{eq:2024july3_9} in the main text).  
From this viewpoint, the non-negativity of the coefficients of the polynomial is obvious\footnote{{The non-negativity of the coefficients of the
above polynomial can also be confirmed 
by proving that it is equivalent to some other polynomial 
that has non-negative coefficients. 
As pointed out by Warnaar \cite{Warnaar2001, Warnaar2004} for instance, it is a special case of the Kostka polynomial. We present a brief review on this subject in Appendix \ref{secA}.}}. 
The above expression of this polynomial as a subtraction of two Gaussian polynomials,
or such expressions for more general polynomials or series with nonnegative coefficients
by formulae including subtractions, are sometimes called bosonic forms.
This terminology {and the corresponding notion of one-dimensional configuration sums
were frequently} 
used in the study of the formulae for
conformal field theory characters and branching {functions
\cite{Schilling1996, SW2002, Warnaar2001}}, and such formulae have their origin in the study of
exactly solvable lattice models in $2D$ statistical mechanics {by 
Andrews, Baxter and Forrester \cite{ABF}.
In this paper we present a derivation of the above configuration sum 
based on our study \cite{Takagi1, Takagi2025}
for the sake of consistency to the subsequent arguments,
but, we emphasize that 
an essentially equivalent formula of the above configuration sum itself was already included in 
their own result \cite[Lemma 2.6.1]{ABF}.}

This relation enables us to verify the second assertion of the above conjecture for $B_n(q)$,
by introducing a partition theoretic function
to specify the unidentified subset of the set of restricted partitions.
This function, denoted by $\mathcal{E}_1$, has its origin in algebraic combinatorics 
or more precisely in
the theory of crystals \cite{KMN},
and plays an essential role in the proof of the main theorem.
To obtain a useful form for this function, we use the notion of the path encodings
of the box-ball system
in a recent study on the integrable cellular automaton \cite{CKST2023}.

The remainder of this paper is organized as follows.
In Sect.~\ref{sec2}, after reviewing
necessary notions about integer partitions and Ferrers diagrams,
we present the definition of our new partition statistics through a combinatorial procedure
and state the main theorem.
Everything what follows is for the proof of this theorem.
In Sect.~\ref{sec3}, we derive formulae for the generating functions of
a pair of mex-related functions and find that the polynomial bosonic form appears there.
In Sect.~\ref{sec4}, we prove that the bosonic form
can be expressed as a statistical configuration sum over the set of bit sequences
with a restriction.
In Sect.~\ref{sec5} and Sect.~\ref{sec6}, 
we define an energy preserving bijection between
restricted partitions and bit sequences, and also define a similar bijection between
partitions and pairs of partitions.
In Sect.~\ref{sec7} we give a proof of the main theorem, and finally
in Sect.~\ref{sec8} we present concluding remarks.
{Some supplementary materials are provided in Appendices \ref{secA}, \ref{secB}, \ref{secC}.}

\section{Preliminaries and the main result}\label{sec2}
Throughout this paper,
we identify partitions with the corresponding Ferrers diagrams, in which 
dots are replaced by unit squares called {\em cells}.
The cell in row $i$ and column $j$ has coordinate $(i,j)$, as in a matrix.
We sometimes use the following notation for a non empty partition $\lambda$ \cite{Macdonald}.
Suppose that the main diagonal of the diagram of $\lambda$ consists of $d$ cells $(i,i) \, (1 \leq i \leq d)$.
Such $d = d(\lambda)$ is called the {\em side of the Durfee square} of $\lambda$ \cite{A}.
Let $x_i$ be the number of cells in the $i$th row of $\lambda$ to the right of $(i,i)$,
and let $y_i$ be the number of cells in the $i$th column of $\lambda$ 
below $(i,i)$, for $1 \leq i \leq d$.
Then we define the {\em Frobenius representation} of $\lambda$ to be $F(\lambda) = (x_1,\dots,x_d \mid y_1,\dots,y_d)$, 
\red{see Figure} \ref{fig5} for an example.
\begin{figure}[h]
\centering
\includegraphics[height=2cm]{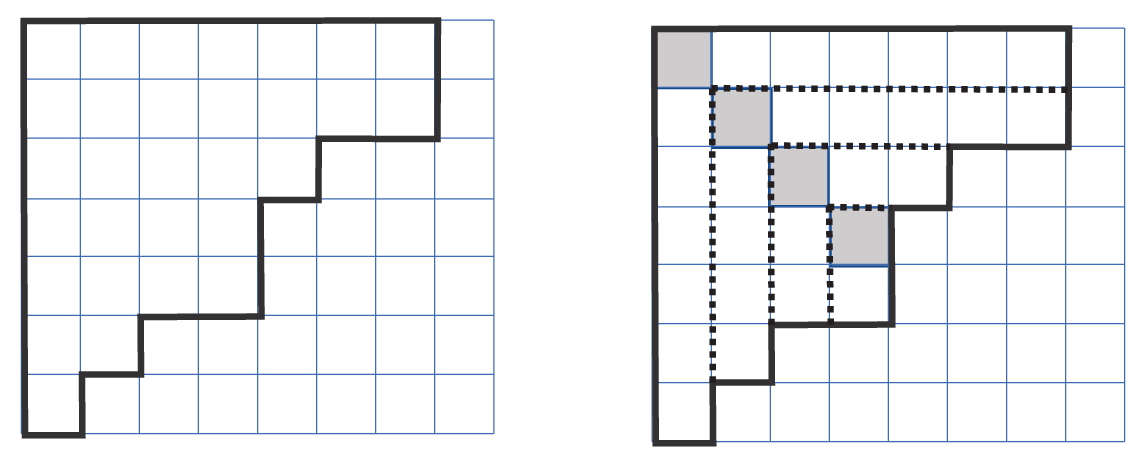}
\caption{(left) The Ferrers diagram for partition $\lambda = (7,7,5,4,4,2,1).$; 
(right) Its decomposition into
the hooks for recognizing its Frobenius representation 
$F(\lambda) =(6,5,2,0 \mid 6,4,2,1)$.}
\label{fig5}
\end{figure}

A {\em rim hook} (also called a border strip) is a skew diagram
if it is edgewise connected and contains no $2 \times 2$ block of cells \cite{Sagan2000}.
The number of all cells contained in a rim hook is called its {\em length}.
By generalizing the corresponding notions for a {\em hook} \cite{Sagan2000}, we define
the {\em arm length} of a rim hook to be one less than the number of columns
it occupies, and the {\em leg length} of a rim hook to be one less than the number of rows
it occupies.
The {\em rim} of a Ferrers diagram $\lambda$ is defined to be the rim hook that
consists of all the cells along the south-east border of $\lambda$.
More precisely, if $F(\lambda) = (x_1,\dots,x_d \mid y_1,\dots,y_d)$, then
the rim of $\lambda$ is defined by the skew shape $\lambda / \nu$ where
$\nu$ is the partition uniquely determined by 
$F(\nu) = (x_2,\dots,x_d \mid y_2,\dots,y_d)$.

Now we introduce the pair of partition statistics or maps
from $\mathscr{P}$ to $\mathbb{Z}_{\geq 0}$,
tentatively denoted by $f^{(a)} \, (a=0,1)$,
by the following procedure.
Given $\lambda \in \mathscr{P}$;
\begin{enumerate}
\item 
Let $D_0(\lambda)$ be the {Durfee square of} $\lambda$, 
and let $D_1(\lambda)$ be the largest rectangle inside $\lambda$ 
{such that the horizontal side length exceeds the vertical side length by one.}
Let $n_a = n_a(\lambda)$ denote the vertical side of $D_a(\lambda)$, and let
$A_a(\lambda)$ denote the sub-diagram of $\lambda$ that consists of 
the cells inside the first $n_a$ rows but not in $D_a(\lambda)$ 
\footnote{If $\lambda = \emptyset$,
we set $D_0(\lambda) = A_0(\lambda)=\emptyset$ and $n_0=0$. 
If $\lambda$ has only ones as parts or $\lambda = \emptyset$,
we set $D_1(\lambda) = A_1(\lambda)=\emptyset$ and $n_1=0$.}. 
\item From diagram $A_a(\lambda)$, 
strip away the longest of the rightmost rim hooks with arm length $n_a + a$
lying along the rim of the diagram.
Repeat the same procedure until the remaining diagram has less than $n_a+a+1$ columns, and
let $R_a(\lambda)$ denote the residual diagram.
\item Let $F(R_a(\lambda))=(x_1,\dots,x_d \mid y_1,\dots,y_d)$.
Set $y_0 = n_a, x_{d+1} = -1$ and define \footnote{If $R_a(\lambda) = \emptyset$, we set $d=0$ and hence $f^{(a)}(\lambda) =n_a$.}
\begin{equation*}
f^{(a)}(\lambda) := \max_{0 \leq i \leq d} (y_i -x_{i+1}) -1.
\end{equation*}
\end{enumerate}
\red{We now introduce our two main new partition statistics.}
\begin{definition}
For \red{$\lambda$ a partition, define}
\begin{enumerate}
\item ${\rm sqrank} (\lambda) = f^{(0)}(\lambda)$,
\item ${\rm rerank} (\lambda) = f^{(1)}(\lambda)$.
\end{enumerate}
\end{definition}

In what follows, we call {$D_1(\lambda)$ 
the {\em Durfee rectangle}} of the partition $\lambda$.

\begin{example}
Let $\lambda = (19,16,9,2,1)$.
Then we have $n_0 = 3, D_0(\lambda) = (3,3,3)$ and
$A_0(\lambda)=(16,13,6)$.
After stripping away the rim hooks with arm length 3 from $A_0(\lambda)$, one has $R_0(\lambda) = (1,1)$.
Then we find $F(R_0(\lambda))=(0 \mid 1)$ and hence
${\rm sqrank}(\lambda)  =\max (y_0-x_1, y_1-x_2) -1
 =\max (3-0, 1-(-1)) -1 = 2$.
See Fig \ref{fig1}.
\begin{figure}[h]
\centering
\includegraphics[height=2cm]{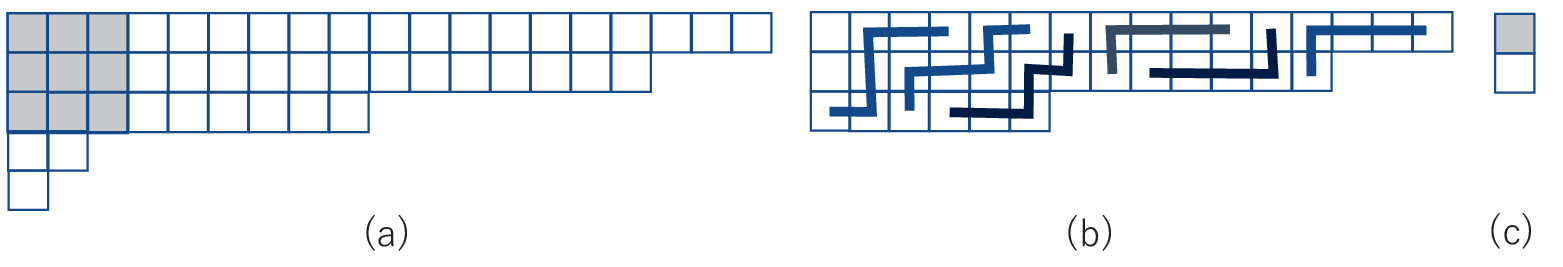}
\caption{(a) The diagram of $\lambda$ and its Durfee square $D_0(\lambda)$ (shaded). 
(b) The rim hooks with arm length 3 that will be stripped away from the sub-diagram $A_0(\lambda)$. 
(c) The residual diagram $R_0(\lambda)$ and its diagonal node (shaded).}
\label{fig1}
\end{figure}

\end{example}

\begin{example}
Let $\lambda = (19,16,9,2,1)$.
Then we have $n_1 = 3, D_1(\lambda) = (4,4,4)$ and
$A_1(\lambda)=(15,12,5)$.
After stripping away the rim hooks with arm length 4 from $A_1(\lambda)$, one has $R_1(\lambda) = (3,3)$.
Then we find $F(R_1(\lambda))=(2,1 \mid 1,0)$ and hence
${\rm rerank}(\lambda) =\max (y_0-x_1, y_1-x_2, y_2-x_3) -1= \max (3-2, 1-1, 0-(-1)) -1 = 0$.
See Fig \ref{fig2}.
\begin{figure}[h]
\centering
\includegraphics[height=1.5cm]{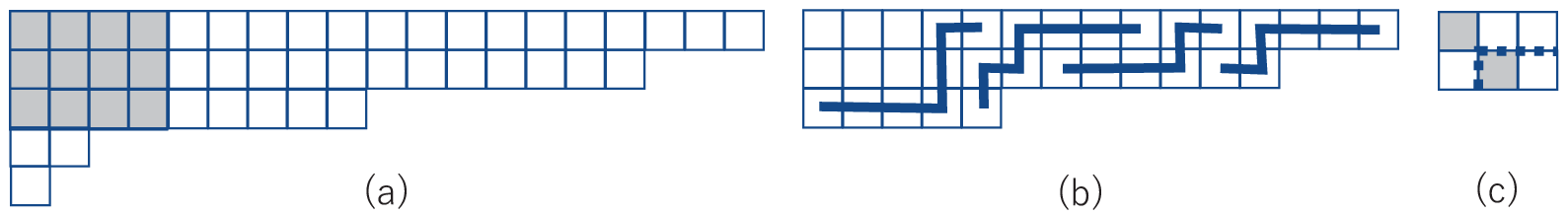}
\caption{(a) The diagram of $\lambda$ and its Durfee rectangle $D_1(\lambda)$ (shaded). 
(b) The rim hooks with arm length 4 that will be stripped away from the sub-diagram $A_1(\lambda)$. 
(c) The residual diagram $R_1(\lambda)$ and its diagonal nodes (shaded), and its decomposition into
the hooks.}
\label{fig2}
\end{figure}

\end{example}

As {introduced 
in Section} \ref{sec1}, ${\rm mex}_{A,a}(\lambda)$ is the smallest positive
integer congruent to $a$ modulo $A$ that is not a part of $\lambda$.
In particular, ${\rm mex}_{2,1}(\lambda)$ (resp.~${\rm mex}_{2,2}(\lambda)$) is the smallest 
positive odd (resp.~even) number that is 
not a part of $\lambda$.
For all $n ,r \in \mathbb{Z}_{\geq 0}$, we define
\begin{align}
p^{(r)}_{2,1}(n) &= \# \{ \lambda \vdash n \mid {\rm mex}_{2,1}(\lambda) = 2r+1 \}, \\
p^{(r)}_{2,2}(n) &= \# \{ \lambda \vdash n \mid {\rm mex}_{2,2}(\lambda) = 2r+2 \},
\end{align}
where $\#$ denotes cardinality.
Here we let $p^{(r)}_{2,1}(0) = p^{(r)}_{2,2}(0) = \delta_{r,0}$ since
for the only partition $\lambda = \emptyset$ of $n=0$ we have ${\rm mex}_{2,1}(\lambda) = 1$
and ${\rm mex}_{2,2}(\lambda) = 2$.
The following is the main result of this paper.
\begin{theorem}\label{th:main}
It holds that
\begin{enumerate}
\item $p^{(r)}_{2,1}(n)$ equals the number of partitions {$\lambda \vdash n$} 
with ${\rm sqrank}(\lambda) =r$.
\item $p^{(r)}_{2,2}(n)$ equals the number of partitions {$\lambda \vdash n$} 
with ${\rm rerank}(\lambda) =r$.
\end{enumerate}
\end{theorem}

\begin{example}
Let $n=8$.
Then, item 1 (resp.~item 2) of Theorem \ref{th:main} is shown in Table \ref{tab1} (resp.~Table \ref{tab2} ).  
In this example, we use the multiplicity notation $\mu_1^{m_1}\mu_2^{m_2}\dots\mu_\ell^{m_\ell}$
to represent a partition
in which $m_i$ denotes the multiplicity of the part
$\mu_i$ (which is omitted when $m_i=1$) and $\mu_1>\mu_2>\dots>\mu_\ell$.
\end{example}

Let $p_{2,2}(n) = \sum_{m \geq 0} p^{(2m)}_{2,2}(n)$, which equals the number of partitions
$\lambda$ of $n$, where
\begin{equation}
{\rm mex}_{2,2}(\lambda) \equiv 2 \pmod{4}.
\end{equation}
As a direct result of the above theorem, we have:
\begin{corollary}\label{coro:2}
$p_{2,2}(n)$ equals the number of partitions of $n$ with even rerank.
\end{corollary}
This {provides an answer to the problem posed} by
Andrews {and} Newman \cite[p.~9, Question II]{AN2020}.

\begin{table}[h]
\caption{Comparison between sqrank and ${\rm mex}_{2,1}$ for the partitions of 8}\label{tab1}
\begin{tabular*}{\textwidth}{@{\extracolsep\fill}lcccccc}
\toprule%
& \multicolumn{3}{@{}c@{}}{sqrank} & \multicolumn{3}{@{}c@{}}{${\rm mex}_{2,1}$} \\\cmidrule{2-4}\cmidrule{5-7}%
Partition & 0 & 1 & 2 & 1 & 3 & 5 \\
\midrule
$1^8$       &                & \ding{"33} &                &                & \ding{"33} &               \\
$21^6$     & \ding{"33} &                &                &                & \ding{"33} &                \\
$2^21^4$  &                &                & \ding{"33} &                & \ding{"33} &                \\
$2^31^2$  &                &                & \ding{"33} &                & \ding{"33} &                \\
$2^4$       &                &                & \ding{"33} & \ding{"33} &                &               \\
$31^5$     &                & \ding{"33} &                &                &                & \ding{"33}\\
$321^3$   &                & \ding{"33} &                &                &                & \ding{"33}\\
$32^21$   &                & \ding{"33} &                &                &                & \ding{"33}\\
$3^21^2$  &                & \ding{"33} &                &                &                & \ding{"33}\\
$3^22$     &                & \ding{"33} &                & \ding{"33}   &                &                \\
$41^4$     & \ding{"33} &                &                &                & \ding{"33} &                \\
$421^2$    & \ding{"33} &                &                &                & \ding{"33} &                \\
$42^2$     & \ding{"33} &                &                &   \ding{"33} &                &                \\
$431$     &                & \ding{"33} &                &                    &                & \ding{"33}\\
$4^2$     & \ding{"33} &                &                &   \ding{"33} &                &                \\
$51^3$   &                & \ding{"33} &                &                & \ding{"33} &               \\
$521$    &                &                & \ding{"33} &                & \ding{"33} &                \\
$53$      &                &                & \ding{"33} & \ding{"33} &                &               \\
$61^2$   & \ding{"33} &                &                &                 &  \ding{"33} &               \\
$62$     &                & \ding{"33} &                & \ding{"33}   &                &                \\
$71$      &                & \ding{"33} &                &                & \ding{"33} &               \\
$8$     & \ding{"33} &                &                   &  \ding{"33}  &                 &                \\
Total & 7 & 10 & 5 & 7 & 10 & 5 \\
\botrule
\end{tabular*}
\end{table}

\begin{table}[h]
\caption{Comparison between rerank and ${\rm mex}_{2,2}$ for the partitions of 8}\label{tab2}
\begin{tabular*}{\textwidth}{@{\extracolsep\fill}lcccccc}
\toprule%
& \multicolumn{3}{@{}c@{}}{rerank} & \multicolumn{3}{@{}c@{}}{${\rm mex}_{2,2}$} \\\cmidrule{2-4}\cmidrule{5-7}%
Partition & 0 & 1 & 2 & 2 & 4 & 6 \\
\midrule
$1^8$       & \ding{"33} &                &                   &  \ding{"33}  &                 &                \\
$21^6$     &                & \ding{"33} &                &                 & \ding{"33} &               \\
$2^21^4$  &                & \ding{"33} &                &                 & \ding{"33} &               \\
$2^31^2$  &                & \ding{"33} &                &                 & \ding{"33} &               \\
$2^4$       &                & \ding{"33} &                &                 & \ding{"33} &               \\
$31^5$     & \ding{"33} &                &                   &  \ding{"33}  &                 &                \\
$321^3$   & \ding{"33} &                &                &                & \ding{"33} &                \\
$32^21$   & \ding{"33} &                &                &                & \ding{"33} &                \\
$3^21^2$  &                &                & \ding{"33} & \ding{"33} &                &               \\
$3^22$     &                &                & \ding{"33} &                & \ding{"33} &                \\
$41^4$     & \ding{"33} &                &                   &  \ding{"33}  &                 &                \\
$421^2$    & \ding{"33} &                &                &                &               &  \ding{"33}   \\
$42^2$     & \ding{"33} &                &                &                &               &  \ding{"33}   \\
$431$     &                & \ding{"33} &                & \ding{"33}   &                &                \\
$4^2$     &                & \ding{"33} &                & \ding{"33}   &                &                \\
$51^3$   &                & \ding{"33} &                & \ding{"33}   &                &                \\
$521$    &                & \ding{"33} &                &                 & \ding{"33} &               \\
$53$      & \ding{"33} &                &                   &  \ding{"33}  &                 &                \\
$61^2$   & \ding{"33} &                &                   &  \ding{"33}  &                 &                \\
$62$     & \ding{"33} &                &                &                & \ding{"33} &                \\
$71$      & \ding{"33} &                &                   &  \ding{"33}  &                 &                \\
$8$     &                & \ding{"33} &                & \ding{"33}   &                &                \\
Total & 11 & 9 & 2 & 11 & 9 & 2 \\
\botrule
\end{tabular*}
\end{table}

\section{Generating functions}\label{sec3}
In this section we establish the following formulae for the 
generating functions of the mex-related functions.
\begin{proposition}\label{prop:2024oct16_2}
The following identities hold:
\begin{align}
\sum_{n \geq 0} p^{(r)}_{2,1}(n) q^n
&=  \sum_{n \geq 0} \frac{q^{n^2} Z_{2n,n}^{(r)}(q)}{(q;q)_{2n}}, \label{eq:2024nov21_1} \\
\sum_{n \geq 0} p^{(r)}_{2,2}(n) q^n 
&=  \sum_{n \geq 0} \frac{q^{n^2+n} Z_{2n+1,n}^{(r)}(q)}{(q;q)_{2n+1}},\label{eq:2024nov21_2}
\end{align}
where
\begin{equation}\label{eq:2024nov22_1}
Z_{2n+a,n}^{(r)}(q) =
{2n+a \brack n-r}
-
{2n+a \brack n-r-1}.
\end{equation}
\end{proposition}

\begin{proof}
First we prove \eqref{eq:2024nov21_1}.
Its right hand side can be written as
\begin{equation*}
\sum_{n \geq 0} \frac{q^{n^2} Z_{2n,n}^{(r)}(q)}{(q;q)_{2n}} =
\sum_{n \geq r} \frac{q^{n^2} }{(q;q)_{2n}} { 2n \brack n-r } \\
- \, ( \quad ``r \rightarrow r + 1" \quad ),
\end{equation*}
where we let the second term be obtained from the first term
by replacing $r$ by $r +1$.
The first term can be written as
\red{
\begin{align}
\sum_{n \geq r} \frac{q^{n^2} }{(q;q)_{2n}} { 2n \brack n-r }  &=
\sum_{n \geq r} \frac{q^{n^2} }{(q;q)_{n-r} (q;q)_{n+r}} \nonumber\\
&=
q^{r^2}\sum_{n \geq r} \frac{q^{(n-r)(n+r)} }{(q;q)_{n-r} (q;q)_{n+r}} \nonumber\\
&= 
\frac{q^{r^2}}{(q;q)_{\infty}}.\label{eq:2025june23_1}
\end{align}
In the last line we used the classical Durfee rectangle identity
\begin{equation}\label{eq:2024july11_3}
\sum_{n \geq r} \frac{q^{(n-r)(n+r+a)} }{(q;q)_{n-r} (q;q)_{n+r+a}} = \frac{1}{(q;q)_{\infty}},
\end{equation}
for any fixed integers $r, a$ satisfying $2r+a \geq 0$.}
Therefore
\begin{align*}
\sum_{n \geq 0} \frac{q^{n^2} Z_{2n,n}^{(r)}(q)}{(q;q)_{2n}} &=
\frac{q^{r^2} - q^{(r+1)^2}}{(q;q)_{\infty}} \\
&= 
\frac{q^{r^2}(1 - q^{2r+1})}{(q;q)_{\infty}} \\
&= \frac{q^{r^2}}{\prod_{m=1, m \ne 2r+1}^\infty (1-q^m)}.
\end{align*}
Since $r^2 = 1 + 3 + \dots + (2 r-1)$, one finds that
this is an expression for the generating function of $p^{(r)}_{2,1}(n)$.

Next we prove \eqref{eq:2024nov21_2}.
Its right hand side can be written as
\begin{equation*}
\sum_{n \geq 0} \frac{q^{n^2+n} Z_{2n+1,n}^{(r)}(q)}{(q;q)_{2n+1}} =
\sum_{n \geq r} \frac{q^{n^2+n} }{(q;q)_{2n+1}} { 2n+1 \brack n-r } \\
- \, ( \quad ``r \rightarrow r + 1" \quad ),
\end{equation*}
where we let the second term be obtained from the first term
by replacing $r$ by $r +1$.
The first term can be written as
\red{
\begin{align}
\sum_{n \geq r} \frac{q^{n^2+n} }{(q;q)_{2n+1}} { 2n+1 \brack n-r }  &=
\sum_{n \geq r} \frac{q^{n^2+n} }{(q;q)_{n-r} (q;q)_{n+r+1}} \nonumber\\
&=q^{r^2+r}\sum_{n \geq r} \frac{q^{(n-r)(n+r+1)} }{(q;q)_{n-r} (q;q)_{n+r+1}} \nonumber\\
&= 
\frac{q^{r(r+1)}}{(q;q)_{\infty}}.\label{eq:2025june23_2}
\end{align}
}
Therefore
\begin{align*}
\sum_{n \geq 0} \frac{q^{n^2+n} Z_{2n+1,n}^{(r)}(q)}{(q;q)_{2n+1}} &=
\frac{q^{r(r+1)} - q^{(r+1)(r + 2)}}{(q;q)_{\infty}} \\
&= 
\frac{q^{r(r+1)}(1 - q^{2(r+1)})}{(q;q)_{\infty}} \\
&= \frac{q^{r(r+1)}}{\prod_{m=1, m \ne 2(r+1)}^\infty (1-q^m)}.
\end{align*}
Since $r(r+1) = 2 + 4 + \dots + 2 r$, one finds that
this is an expression for the generating function of $p^{(r)}_{2,2}(n)$.
\end{proof}

\begin{remark}
By summing both sides of \eqref{eq:2024nov21_2} over even $r$'s,
we obtain
\begin{equation}\label{eq:2024dec25_1}
\sum_{n \geq 0} p_{2,2}(n) q^n 
=  \sum_{n \geq 0} \frac{q^{n^2+n} B_n(q)}{(q;q)_{2n+1}},
\end{equation}
where
\begin{equation}\label{eq:2024july2_1}
B_n(q) = \sum_{r=0}^n (-1)^r {2n+1 \brack n-r},
\end{equation}
which we have already mentioned in Sect.~\ref{sec1}.
This expression for the generating function of $p_{2,2}(n)$ was first derived 
through a different argument by
Andrews and Newman \cite{AN2020}.
\end{remark}

{
\begin{remark}
Based on the fact that the polynomials defined in \eqref{eq:2024nov22_1} are
Kostka polynomials \cite{Warnaar2001, Warnaar2004},
the identities \eqref{eq:2024nov21_1} and \eqref{eq:2024nov21_2} 
can be viewed as affine Lie algebra 
character formulae. We present a brief review on this subject in Appendix \ref{secB}.
\end{remark}
\begin{remark}
Consider the sequences $\{ \gamma_L \}_{L \geq 0}, \{ \delta_L \}_{L \geq 0}$
given by
\begin{equation*}
\gamma_L = \frac{a^L q^{L^2}}{(aq;q)_\infty} \quad \mbox{and} \quad
\delta_L = a^L q^{L^2}.
\end{equation*}
This pair of sequences is known as a conjugate 
Bailey pair \cite{Warnaar2001, SW2002} relative to $a$,
and satisfies
\begin{equation*}
\gamma_L = \sum_{r=L}^\infty \frac{\delta_r}{(q;q)_{r-L} (aq;q)_{r+L}}. 
\end{equation*}
Both \eqref{eq:2025june23_1} and \eqref{eq:2025june23_2} are special cases of
this equation, with $a=1$ and $a=q$ respectively.
\end{remark}
}

\section{A polynomial bosonic form for statistical configuration sums}\label{sec4}
In this section we present {a derivation for} 
the formulae for the statistical configuration sums
associated with an integrable cellular automaton
which the author studied some time ago \cite{Takagi1}, 
with an argument for proving them being supplied
\footnote{Part of the following results are also presented in Japanese \cite{Takagi2025}.}.
{Although an essentially equivalent result had been obtained much earlier 
by Andrews, Baxter and Forrester \cite{ABF}\footnote{We provide a brief explanation on this subject in Appendix  \ref{secC}.}, 
we present our own one 
because our derivation here well fits the definition
of the energy \eqref{eq:2024july3_9} defined in the next subsection and hence is suitable for
the subsequent arguments.}
\subsection{A linear recursion relation associated with the Gaussian polynomials}
We seek a family of polynomials $G_{L,s}(q)$ satisfying
the recursion relation
\begin{equation} \label{eq:may13_2}
G_{L,s}(q) = G_{L-1,s}(q) + \sum_{k=1}^s q^{L-k}
G_{L-k-1,s-k}(q),
\end{equation}
for $s<L$.

\begin{proposition}\label{prop:2025oct24_1}
\red{For $L$ a nonnegative integer, let 
\begin{equation} \label{eq:may13_2x_20250ct22}
G_{L-1,L}(q) = \delta_{L,0}, 
\end{equation}
and consider the recursion relation \eqref{eq:may13_2} with this boundary condition.
Then it determines a unique polynomial $G_{L,s}(q) =: X_{L,s}(q)$ for every $0 \leq s \leq L$.}
This polynomial has nonnegative coefficients.
\end{proposition}

\begin{proposition}\label{prop:2025oct24_2}
\red{For $L$ a nonnegative integer, let 
\begin{equation} \label{eq:may13_2xx_20250ct22}
G_{2L-1,L}(q) = \delta_{L,0}, 
\end{equation}
and consider the recursion relation \eqref{eq:may13_2} with this boundary condition.
Then it determines a unique polynomial $G_{L,s}(q) =: X^+_{L,s}(q)$ for every
$0 \leq s \leq \lfloor L/2 \rfloor$.}
This polynomial has nonnegative coefficients.
\end{proposition}
\red{
\begin{proof}
Since $G_{L,s}(q)$ for $0 \leq s \leq L$ (resp.~$0 \leq s \leq \lfloor L/2 \rfloor$)
is expressed as a totally positive linear combination of 
$G_{L',s'}(q)$ with $L'<L$ and $0\leq s'\leq s$ such that $L'-s'\geq -1$
(resp.~$L'-2s'\geq -1$), the claim of Proposition \ref{prop:2025oct24_1} 
(resp.~Proposition \ref{prop:2025oct24_2} ) follows
by induction on $s$ and $L$.
\end{proof}
}

Here we summarize the definition of the above polynomials.
\begin{definition}\label{def:2}
Let $X_{L,s}(q)$ and $X^+_{L,s}(q)$ denote the polynomials with
nonnegative coefficients that are determined by;
\begin{enumerate}
\item Recursion relation \eqref{eq:may13_2} with boundary condition \eqref{eq:may13_2x_20250ct22} for $X_{L,s}(q)$.
\item Recursion relation \eqref{eq:may13_2} with boundary conditions \eqref{eq:may13_2xx_20250ct22} for $X^+_{L,s}(q)$.
\end{enumerate}
\end{definition}

These polynomials can be expressed by using Gaussian polynomials \eqref{eq:gaussianpoly}.

\begin{proposition}\label{prop:6}
These polynomials take the following forms
\begin{align}
X_{L,s}(q) &=
{ L \brack s }, \label{eq:may14_8}\\
X^+_{L,s}(q) &=
{ L \brack s }- { L \brack s-1}. \label{eq:may17_1}
\end{align}
\end{proposition}
\begin{proof}
\red{
First we note that by Definition \ref{def:2} we have
\begin{equation} \label{eq:may13_2x}
X_{L,L}(q) = 1 \qquad \mbox{for} \qquad L \geq 0, 
\end{equation}
and 
\begin{equation} \label{eq:may13_2xx}
X^+_{0,0}(q) =1, \quad X^+_{2s-1,s}(q) = 0 \qquad \mbox{for} \qquad s \geq 1.
\end{equation}
We adopt \eqref{eq:may13_2x} and \eqref{eq:may13_2xx} as alternative
boundary conditions for \eqref{eq:may13_2x_20250ct22} and \eqref{eq:may13_2xx_20250ct22}, since
the recursion relation \eqref{eq:may13_2} with these conditions
also uniquely determines the polynomials $X_{L,s}(q)$ and $X^+_{L,s}(q)$. }

The Gaussian polynomials satisfy the following identities \cite{A}:
\begin{align}
{ L \brack s }  &= { L \brack L-s }, \label{eq:2024july3_1} \\ 
{ L \brack s } &={ L-1 \brack s } + q^{L-s} { L-1 \brack s-1 }. \label{eq:2024july3_2}
\end{align}
First we prove \eqref{eq:may14_8}.
By definition the Gaussian polynomials satisfy the boundary condition \eqref{eq:may13_2x}.
\red{Let $1 \leq s \leq L-1$.
For the identity (\ref{eq:2024july3_2}), we see that the Gaussian polynomial in
the second term of its right hand side can
be written as
\begin{align*}
{ L-1 \brack s-1 }  &= { L-1 \brack L-s } \nonumber\\
&= \sum_{j=0}^{s-1} q^j { L-s-1+j \brack L-s-1 } \nonumber\\
&= \sum_{k=1}^{s} q^{s-k} { L-k-1 \brack L-s-1 } \nonumber\\
&= \sum_{k=1}^{s} q^{s-k} { L-k-1 \brack s-k },
\end{align*}
where in the second line we used another formula (\cite[Theorem 3.4, (3.3.9)]{A})
related to the Gaussian polynomials. 
}
%
\red{
As a result, we obtain 
\begin{equation*}
{ L \brack s } ={ L-1 \brack s } + \sum_{k=1}^{s} q^{L-k} { L-k-1 \brack s-k },
\end{equation*}
namely the recursion relation \eqref{eq:may13_2}.}
Since $X_{L,s}(q)$ is uniquely determined by
the boundary condition and the recursion relation, 
we get the expression \eqref{eq:may14_8}. 

Next we prove \eqref{eq:may17_1}.
\red{By replacing $s$ by $s-1$ in the above result, we obtain
\begin{equation*}
{ L \brack s-1 } ={ L-1 \brack s-1 } + \sum_{k=1}^{s} q^{L-k} { L-k-1 \brack s-k-1 },
\end{equation*}
where we extended the upper bound for the summation to $k=s$
by the null condition in the definition of Gaussian polynomials \eqref{eq:gaussianpoly}.
Then, by the linearity of the recursion relation,
the right hand side of \eqref{eq:may17_1} must satisfy the
same recursion relation \eqref{eq:may13_2}.
}
In addition, by that \red{null condition} and
the identity \eqref{eq:2024july3_1}, 
we see that the right hand side of \eqref{eq:may17_1} satisfies 
the boundary condition \eqref{eq:may13_2xx}.
Since $X^+_{L,s}(q)$ is also uniquely determined by
the boundary condition and the recursion relation, 
we get the expression \eqref{eq:may17_1}.
\end{proof}


\subsection{Canonical partition functions: statistical configuration sums over bit sequences}\label{sec4_2}
We consider the bit sequences of
$L-s$ zeros and $s$ ones.
More explicitly we define
\begin{equation}
\mathcal{H} (L,s) = \{ \eta \in \{ 0,1 \}^L \, \mid \, |\eta| = s \},
\end{equation}
where for $\eta = (\eta_1, \dots , \eta_L)\in \{ 0,1 \}^L$
we let $|\eta| =\sum_{1 \leq i \leq L} \eta_i$.
In what follows,
we sometimes write $(\eta_1, \dots , \eta_L)$ as $\eta_1 \dots \eta_L$.

Let $H$ denote the function from $ \{0,1\}^2$ to $\{0,1\}$ that is given by
$H(0,1)=1, H(0,0)=H(1,0)=H(1,1)=0$.
For bit sequences $\eta \in \mathcal{H} (L,s)$, we define their energy by
\begin{equation}\label{eq:2024july3_9} 
E(\eta) = \sum_{1 \leq j \leq L-1} j H(\eta_j, \eta_{j+1}).
\end{equation}
We define the generating function enumerating
the bit sequences with the energy as
\begin{equation}
Z_{L,s}(q) =
\sum_{\eta \in \mathcal{H} (L,s)} q^{E(\eta)},
\end{equation}
\red{for $L \geq 1$, and let $Z_{0,0}(q) =1$ \footnote{Accordingly, we formally define 
$\mathcal{H} (0,0)$ to be the set 
that has only one element with energy $0$.}.}
Borrowing a terminology from statistical mechanics, we
call this generating function a {\em partition function}.

\begin{example}
Let $L=5,s=2$.
The ten elements of $\mathcal{H} (5,2)$ are 
$11000 \, (E=0)$, $01100, \, (E=1)$, $00110, 10100 \, (E=2)$, 
$00011, 10010 \, (E=3)$, $01010, 10001 \, (E=4)$,
$01001 \, (E=5)$, and $00101 \, (E=6)$.  
Therefore we have $Z_{5,2}(q) =1 + q + 2q^2+2q^3+2q^4+ q^5+q^6$.
\end{example}

\begin{proposition}\label{prop:2024nov27_1}
\begin{equation}
Z_{L,s}(q)  =
{ L \brack s }. \label{eq:2024july3_5}\\
\end{equation}
\end{proposition}
\begin{proof}
By Proposition \ref{prop:6}, it suffices to show that $Z_{L,s}(q)  =X_{L,s}(q)$.
\red{Since the only one element that belongs to $\mathcal{H} (L,L)$ 
has zero energy, }
$Z_{L,s}(q)$ satisfies the boundary condition \eqref{eq:may13_2x}
and the remaining task for us is to verify the recursion relation \eqref{eq:may13_2}.

By classifying the elements of $\mathcal{H} (L,s)$ by the positions $L-k$ of the last zero, we have
\begin{align*}
\mathcal{H} (L,s) &= \coprod_{k = 0}^s \mathcal{H}_{(k)} (L,s),\\
\mathcal{H}_{(k)} (L,s) &:= \{ \eta \in \mathcal{H} (L,s) \, \mid \, 
\max \{ \alpha \, \mid \, \eta_\alpha=0 \} = L-k \}. 
\end{align*}
According to this disjoint union, the partition function takes the form
\begin{equation*}
Z_{L,s}(q) =
\sum_{\eta \in \mathcal{H}_{(0)}  (L,s)} q^{E(\eta)} +
\sum_{k=1}^s \sum_{\eta \in \mathcal{H}_{(k)}  (L,s)} q^{E(\eta)}.
\end{equation*}
In the first term of the right hand side,
we have 
$E(\eta) =\sum_{1 \leq j \leq L-2} j H(\eta_j, \eta_{j+1})$, because for every
$\eta \in \mathcal{H}_{(0)}  (L,s)$ we have
$H(\eta_{L-1},\eta_L) =0$.
In addition, we have $\sum_{1 \leq i \leq L-1}\eta_i = s$.
This implies that the procedure of removing the last digit $0$ induces an energy preserving
bijection from 
$\mathcal{H}_{(0)}  (L,s)$ to $\mathcal{H}  (L-1,s)$.
Therefore, this first term equals to the first term of the right hand side of \eqref{eq:may13_2}.
In the second term of the right hand side, for each $k$,
we have 
$E(\eta) = L-k + \sum_{1 \leq j \leq L-k-1} j H(\eta_j, \eta_{j+1})$, because for every
$\eta \in \mathcal{H}_{(k)}  (L,s)$ we have
$H(\eta_{L-k+j},\eta_{L-k+j+1}) =\delta_{j,0}$ for $0 \leq j \leq k-1$.
In addition, we have $\sum_{1 \leq i \leq L-k-1}\eta_i = s-k$.
This implies that the procedure of removing the last $k+1$ digits $01\dots1$
induces a bijection from 
$\mathcal{H}_{(k)}  (L,s)$ to $\mathcal{H}  (L-1-k,s-k)$ that reduces 
the value of the energy by $L-k$.
Therefore, this second term as a whole 
equals to the second term of the right hand side of \eqref{eq:may13_2}.
\end{proof}

For $\eta \in \mathcal{H} (L,s)$ let
$S(\eta) = (S_i (\eta))_{0 \leq i \leq L}$ be its {\em path-encoding} \cite{CKST2023}
that is defined as
\begin{align}
S_0 (\eta) &= 0,  \nonumber\\
\quad S_i (\eta) &= S_{i-1} (\eta) +1 - 2 \eta_i \,(1 \leq i \leq L). \label{eq:2024nov27_4}
\end{align}
See Fig~\ref{fig:pathenc2} for an example.
In what follows we call $S(\eta)$ the {\em path} for the bit sequence $\eta$.

\begin{figure}[h]
\centering
\includegraphics[height=2cm]{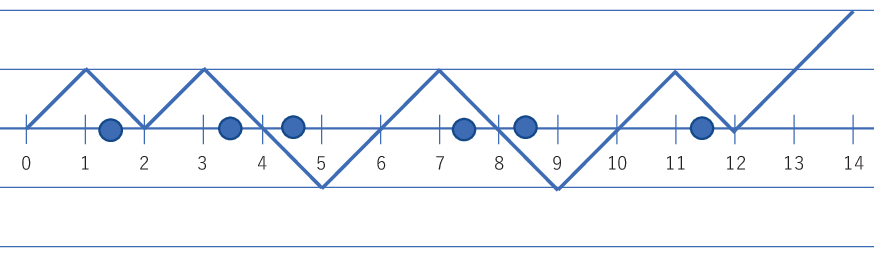}
\caption{A graphical representation of the path $S(\eta)$ for 
$\eta=01011001100100$ by line segments.
We put a mark $\bullet$ between coordinate $i-1$ and $i$ if $\eta_i = 1$.}
\label{fig:pathenc2}
\end{figure}

Let $\varepsilon_1(\eta)$ be the absolute value of
the minimum of the path $S(\eta)$, i.e. 
\begin{equation}\label{eq:2024nov27_5}
\varepsilon_1(\eta) := - \min_{0 \leq i \leq L} S_i(\eta),
\end{equation}
by which we also define the function $\varepsilon_1: \mathcal{H} (L,s) \rightarrow [0,s].$
By classifying the elements of $\mathcal{H} (L,s)$ by their $\varepsilon_1$-values,
we have
\begin{align}
\mathcal{H} (L,s) &= \coprod_{r = 0}^s \mathcal{H}^{(r)} (L,s),\label{eq:2024july3_10} \\
\mathcal{H}^{(r)} (L,s) &:= \{ \eta  \mid
\varepsilon_1(\eta) = r \}. \nonumber 
\end{align}

\begin{proposition}\label{prop:5}
For every integer $L \geq 2$, 
and for every pair of integers $r,s$ satisfying the conditions
$0 \leq s \leq \lfloor L/2 \rfloor -1, 0 \leq r \leq s$,
there is an energy preserving bijection between $\mathcal{H}^{(r)} (L,s)$ and
$\mathcal{H}^{(r + 1)} (L,s+1)$ with respect to the energy \eqref{eq:2024july3_9}.
\end{proposition}

\begin{proof}
We construct the bijection explicitly.
See Fig \ref{fig:pathenc} for an example.

First we construct the map from $\mathcal{H}^{(r)} (L,s)$ to 
$\mathcal{H}^{(r + 1)} (L,s+1)$, denoted by $\tilde{f}_1$, as follows.
For all $\eta \in \mathcal{H}^{(r)} (L,s)$, we see that
if $S_i(\eta) = -r$ then $i \leq L-2$ because we have 
$S_L(\eta) =L-2s \geq 2$, and hence $S_{L-1}(\eta) \geq 1$.
\begin{itemize}
\item If $\max \{ i \mid S_i(\eta) = -r \} = 0$, then $r=0$.
In this case we have $(\eta_1, \eta_2) = (0,0)$.
Replace it by $(1,0)$.
\item If $i:=\max \{ j \mid S_j(\eta) = -r \} > 0$, then we have
$(\eta_i, \eta_{i+1},\eta_{i+2}) = (1,0,0)$.
Replace it by $(1,1,0)$.
\end{itemize}
By this procedure, whereas the number of \red{$1$s} becomes $s+1$ but the energy of 
the bit sequence does not change.
Although the values of $S_j$ for $j>i$ decrease by $2$,
they are greater than or equal to $-r-1$, and in particular we have $S_{i+1} (\tilde{f}_1 \eta)= -r -1$.
On the other hand, the values of $S_j$ for $j \leq i$ do not change, and hence are
greater than or equal to $-r$. 

Next we construct the map from $\mathcal{H}^{(r + 1)} (L,s+1)$ to $\mathcal{H}^{(r)} (L,s)$, 
denoted by $\tilde{e}_1$, as follows.
For all $\eta \in \mathcal{H}^{(r+1)} (L,s+1)$, we see that
if $S_i(\eta) = -r-1$ then $1 \leq i \leq L-1$, because we have 
$S_L(\eta) =L-2s-2 \geq 0$, and $S_0(\eta) =0$.
\begin{itemize}
\item If $\min \{ i \mid S_i(\eta) = -r-1 \} = 1$, then $r=0$.
In this case we have $(\eta_1, \eta_2) = (1,0)$.
Replace it by $(0,0)$.
\item If $i:=\min \{ j \mid S_j(\eta) = -r-1 \} > 1$, then we have
$(\eta_{i-1}, \eta_{i},\eta_{i+1}) = (1,1,0)$.
Replace it by $(1,0,0)$.
\end{itemize}
By this procedure, whereas the number of \red{$1$s} becomes $s$ but the energy of 
the bit sequence does not change.
Since the values of $S_j$ for $j \geq i$ increase by $2$,
they are greater than or equal to $-r+1$.
On the other hand, the values of $S_j$ for $j < i$ do not change, and hence are
greater than or equal to $-r$, and in particular we have $S_{i-1} (\tilde{e}_1 \eta)= -r$. 

It is also easy to see that the maps $\tilde{f}_1$ and $\tilde{e}_1$ are inverse to each other .
\end{proof}
\begin{figure}[htbp]
\centering
\includegraphics[height=5cm]{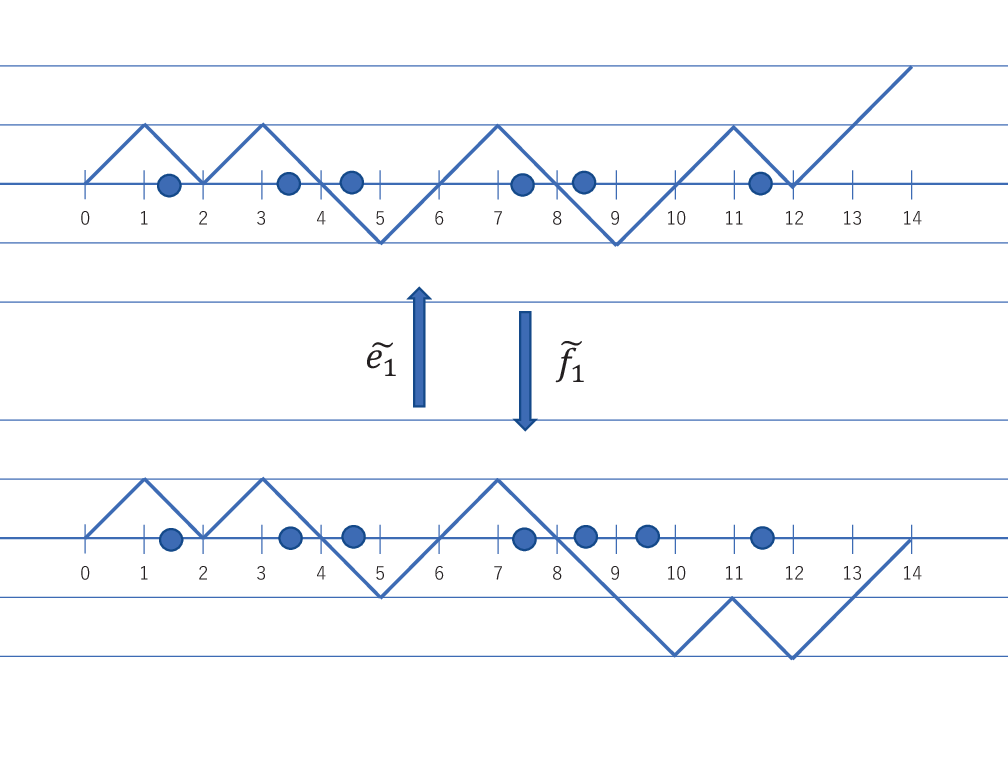}
\caption{A graphical representation of maps $\tilde{f}_1$ and $\tilde{e}_1$.
We have $\tilde{f}_1(01011001100100) = 01011001110100 \in \mathcal{H}^{(2)} (14,7)$ and $\tilde{e}_1(01011001110100) = 01011001100100 \in \mathcal{H}^{(1)} (14,6)$.
The values of the energy coincide as $H(01011001100100) = H(01011001110100) = 1+3+7+11=22$.}
\label{fig:pathenc}
\end{figure}

\begin{remark}
The notation $\tilde{f}_1, \tilde{e}_1$ for the maps in the proof of Proposition \ref{prop:5}
is so chosen as to coincide with the notation for the Kashiwara operators in the theory of
affine crystals. 
The notation $\varepsilon_1$ for the map defined in \eqref{eq:2024nov27_5} is also taken from
this theory.
From this point of view, the set of bit sequences $\mathcal{H} (L,s)$ is
regarded as a subset of $B^{\otimes L}$, where $B = \{ 0,1 \}$ is a classical crystal.
The energy \eqref{eq:2024july3_9} for bit sequences is used to interpret
the sequences as elements of an affine crystal. 
The function $H$ (but the sign has been changed) is called an energy function \cite{KMN}.
\end{remark}

According to \eqref{eq:2024july3_10}, we define the partition function with
fixed $\varepsilon_1$-values by
\begin{equation}\label{eq:2024oct18_1}
Z_{L,s}^{(r)}(q)=
\sum_{\eta \in \mathcal{H}^{(r)} (L,s)} q^{E(\eta)},
\end{equation}
\red{for $L >s \geq r \geq 0$, and let $Z_{0,0}^{(0)}(q) =1$.}

We show that it
is consistent with notation in \eqref{eq:2024nov22_1}.

\begin{proposition}\label{prop:10}
\begin{equation}\label{eq:2024july3_6}
Z_{L,s}^{(r)}(q) = 
{ L \brack s -r}- { L \brack s-r-1}. 
\end{equation}
\end{proposition}
\begin{proof}
By Proposition \ref{prop:5}, we have $Z_{L,s}^{(r)}(q) = Z_{L,s-r}^{(0)}(q)$ for $1 \leq r \leq s$.
Thus it suffices to establish this formula in the case of $r=0$.
Then by Proposition \ref{prop:6}, it suffices to show that $Z^{(0)}_{L,s}(q)  =X^{+}_{L,s}(q)$.
\red{For $s \geq 1$, we have $\mathcal{H}^{(0)} (2s-1,s) = \emptyset$, because the condition
$S_{2s-1}(\eta) = -1$ holds for every $\eta \in \mathcal{H}(2s-1,s)$.
Recalling our convention $Z_{0,0}^{(0)}(q) =1$ as well,  we see that}
$Z^{(0)}_{L,s}(q)$ satisfies the boundary condition \eqref{eq:may13_2xx}
and the remaining task for us is to verify the recursion relation \eqref{eq:may13_2}.

In what follows, let $\mathcal{H}^{+} (L,s) = \mathcal{H}^{(0)} (L,s)$.
By classifying the elements of $\mathcal{H}^{+} (L,s)$ by the positions $L-k$ of 
the last zero, we have
\begin{align*}
\mathcal{H}^{+} (L,s) &= \coprod_{k = 0}^s \mathcal{H}^{+}_{(k)} (L,s),\\
\mathcal{H}^{+}_{(k)} (L,s) &:= \{ \eta \in \mathcal{H}^{+} (L,s) \mid
\max \{ \alpha \mid \eta_\alpha=0 \} = L-k \}. 
\end{align*}
According to this disjoint union, the partition function takes the form
\begin{equation*}
Z^{(0)}_{L,s}(q) =
\sum_{\eta \in \mathcal{H}^{+}_{(0)}  (L,s)} q^{E(\eta)} +
\sum_{k=1}^s \sum_{\eta \in \mathcal{H}^{+}_{(k)}  (L,s)} q^{E(\eta)}.
\end{equation*}
By using the same argument in the proof of Proposition \ref{prop:2024nov27_1},
we see that the right hand side of this expression
equals to the right hand side of \eqref{eq:may13_2} with $G_{L,s}(q) =Z^{(0)}_{L,s}(q)$.
The only subtlety may occur when $L=2s$, but in this case we always have $\eta_L = 1$ and hence 
$\mathcal{H}^{+}_{(0)}  (2s,s) = \emptyset$, which is consistent with
$Z^{(0)}_{2s-1,s}(q)=0$.
\end{proof}
{
\begin{remark}
Proposition \ref{prop:10} in the case of $r=0$ can also be confirmed by
verifying that both sides of equation \eqref{eq:2024july3_6} are equivalent to
the same Kostka polynomial $K_{\lambda, \mu}(q)$ with
$\lambda =2^s 1^{L-2s}$ and $\mu =1^L$ (Appendix \ref{secA}).
\end{remark}
}

\begin{example}
Let $L=5, s=2$.
Then we have
\begin{align*}
\mathcal{H}^{(0)} (5,2) &= \{ 00110, 00011, 01010, 01001, 00101 \}, \\
\mathcal{H}^{(1)} (5,2) &= \{ 01100, 10100, 10010, 10001 \}, \\
\mathcal{H}^{(2)} (5,2) &= \{ 11000 \}. 
\end{align*}
One finds that the corresponding partition functions indeed take the forms
\begin{align*}
Z^{(0)}_{5,2}(q) &={ 5 \brack 2}- { 5 \brack 1}  = q^2 + q^3 + q^4 + q^5 + q^6, \\
Z^{(1)}_{5,2}(q) &={ 5 \brack 1}- { 5 \brack 0}  = q + q^2 + q^3 + q^4, \\
Z^{(2)}_{5,2}(q) &={ 5 \brack 0}- { 5 \brack -1}  = 1. 
\end{align*}

\end{example}

\section{An energy preserving bijection between restricted partitions and bit sequences}\label{sec5}
In this section we define an energy preserving bijection between the sets of restricted partitions 
and bit sequences.
Together with the result in Sect.~\ref{sec4}, this enables us to represent the function
$Z_{2n+a,n}^{(r)}(q)$ in \eqref{eq:2024nov22_1} as a statistical configuration sum over
a subset of the restricted partitions.

We let $\mathscr{P}^{(n)}_l (\subset \mathscr{P})$ denote 
the integer partitions having at most $n$ parts 
and each part being less than or equal to $l$.
The number of cells contained in $\lambda$ is denoted by $|\lambda|$,
and called the {\em energy} of $\lambda$.
Recall that the energy of a bit sequence was defined by \eqref{eq:2024july3_9}.
The following result is a slight modification of those in the appendix of
the author's former work \cite{Takagi1}.

\begin{proposition}\label{prop:11}
There is an energy preserving bijection between $\mathcal{H}(L,s)$ and $\mathscr{P}^{(s)}_{L-s}$.
\end{proposition}
\begin{proof}
We construct the bijection explicitly.
First we define the map $\Phi : \mathcal{H}(L,s) \rightarrow \mathscr{P}^{(s)}_{L-s}$
as follows.
For 
\begin{equation}\label{eq:2024nov27_2}
\eta = \eta_{\emptyset} :=(\underbrace{1,\dots,1}_{s}, \underbrace{0, \dots, 0}_{L-s}),
\end{equation}
we set $\Phi(\eta) = \emptyset$.
Otherwise, every bit sequence $\eta \in \mathcal{H}(L,s)$ has the expression
\begin{equation}\label{eq:2024july5_1}
\eta =(\underbrace{1,\dots,1}_{\alpha_1}, \underbrace{0, \dots, 0}_{\beta_1},
\underbrace{1,\dots,1}_{\alpha_2}, \underbrace{0, \dots, 0}_{\beta_2},\dots \dots,\underbrace{1,\dots,1}_{\alpha_{d+1}}, \underbrace{0, \dots, 0}_{\beta_{d+1}}),
\end{equation}
where $d \geq 1, \sum_{k=1}^{d+1} \alpha_k = s$, and $\sum_{k=1}^{d+1} \beta_k = L-s$. 
Here $\alpha_1$ and $\beta_{d+1}$ are non-negative, and the other 
$\alpha_k$'s and  $\beta_k$'s are positive.
For $1 \leq i \leq d$, let
\begin{equation}\label{eq:2025jan10_1}
x_i = \sum_{j=1}^{d+1-i} \beta_j -1, \quad y_i = \sum_{j=1}^{d+1-i} \alpha_j.
\end{equation}
Then we have 
\begin{equation}\label{eq:2024nov27_3}
L-s-1 \geq x_1 > \dots > x_d \geq 0 \quad
\mbox{and} \quad s-1 \geq y_1 > \dots > y_d \geq 0.
\end{equation}
Let $\Phi (\eta)$ be the partition that has the Frobenius representation 
$F(\Phi (\eta)) = (x_1,\dots,x_d \mid y_1, \dots, y_d)$.
Then clearly one has $\Phi (\eta) \in \mathscr{P}^{(s)}_{L-s}$.

The inverse map can be easily constructed.
For $\lambda = \emptyset \in \mathscr{P}^{(s)}_{L-s}$, we define $\Phi^{-1}(\emptyset)$ 
to be $\eta_{\emptyset}$ given by \eqref{eq:2024nov27_2}.
Otherwise every $\lambda \in \mathscr{P}^{(s)}_{L-s}$ admits a Frobenius representation
of the form 
\begin{equation}\label{eq:2024dec5_1}
F(\lambda) = (x_1,\dots,x_d \mid y_1, \dots, y_d),
\end{equation}
satisfying the condition
\eqref{eq:2024nov27_3}, where $d = d(\lambda) \geq 1$ is the side of the Durfee square of $\lambda$.
For this $\lambda$ we set
\begin{equation}\label{eq:2024oct16_1}
\Phi^{-1}(\lambda)  =(\underbrace{1,\dots,1}_{y_d}, \underbrace{0, \dots, 0}_{x_d+1},
\underbrace{1,\dots,1}_{y_{d-1}-y_d}, \underbrace{0, \dots, 0}_{x_{d-1}-x_d},\dots \dots,\underbrace{1,\dots,1}_{y_{1}-y_2}, \underbrace{0, \dots, 0}_{x_{1}-x_2}
\underbrace{1,\dots,1}_{s-y_1}, \underbrace{0, \dots, 0}_{L-s-x_1-1}).
\end{equation}
Then we see that $\Phi^{-1}(\lambda) \in \mathcal{H}(L,s)$, and
that it reproduces \eqref{eq:2024july5_1} by plugging the $x_i$'s and $y_i$'s in \eqref{eq:2025jan10_1}
into \eqref{eq:2024oct16_1}.
Furthermore, by looking over this bit sequence from the left to the right,
and summing up the positions at which the $0$s turn into the $1$s, 
one can verify the energy preservation
\begin{equation*}
E(\Phi^{-1}(\lambda)) = |\lambda| = \sum_{i=1}^d (x_i + y_i) + d.
\end{equation*}
The proof is completed.
\end{proof}

\begin{example}
Let $L=15, s=7$ and $\eta = 101001100011010$.
Then we have $\alpha_1 = \alpha_2= 1, \alpha_3 = \alpha_4 = 2, \beta_1 = 1, \beta_2=2, \beta_3=3$,
and $\beta_4=1$.
This implies that
\begin{align*}
x_1 &= \beta_1 + \beta_2 + \beta_3 + \beta_4 - 1 = 6, &\quad
y_1 &= \alpha_1 + \alpha_2 + \alpha_3 + \alpha_4 = 6, \\
x_2 &= \beta_1 + \beta_2 + \beta_3 - 1 = 5, &\quad
y_2 &= \alpha_1 + \alpha_2 + \alpha_3 = 4, \\
x_3 &= \beta_1 + \beta_2 - 1 = 2, &\quad
y_3 &= \alpha_1 + \alpha_2  = 2, \\
x_4 &= \beta_1  - 1 = 0, &\quad
y_4 &= \alpha_1   = 1.
\end{align*}
Therefore we have $F(\Phi (\eta)) = (6,5,2,0 \mid 6,4,2,1)$ and hence
$\Phi (\eta) = (7,7,5,4,4,2,1)$.
(See Fig.~\ref{fig5}.)
The energy preservation is confirmed as $E(\eta) = 2+5+10+13 = 30 = 7+7+5+4+4+2+1 =|\Phi (\eta)|$.

\end{example}

Recall the definitions of the path-encoding \eqref{eq:2024nov27_4}
and the function $\varepsilon_1$ \eqref{eq:2024nov27_5}.
In the bit sequence \eqref{eq:2024oct16_1},
find the positions at which the $1$s turn into the $0$s.
By noting that they are the positions at which the path $S(\Phi^{-1}(\lambda))$ 
takes on its local minima (See Fig.~\ref{fig:pathenc2} for an example),
we can obtain the following expression
\begin{equation*}
\varepsilon_1 \circ \Phi^{-1}(\lambda) = \max_{0 \leq i \leq d(\lambda)} (y_i -x_{i+1}) -1
=: \mathcal{E}_1 (\lambda; s),
\end{equation*}
where we set $y_0 = s$ and $x_{d(\lambda)+1}=-1$.
Here we also introduced the function $\mathcal{E}_1$ as:
\begin{definition}
Define the function $\mathcal{E}_1$
from $\mathscr{P}$ to $\mathbb{Z}_{\geq 0}$ by
\begin{equation}\label{eq:2024nov27_6}
\mathcal{E}_1 (\lambda; n) =  \max_{0 \leq i \leq d} (y_i -x_{i+1}) -1,
\end{equation}
where $n \in \mathbb{Z}_{\geq 0}$ is a fixed parameter,
$( x_i, y_i )_{1 \leq i \leq d}$ are the entries of the Frobenius representation \eqref{eq:2024dec5_1},
$y_0 = n$ and $x_{d+1}=-1$.
\end{definition}
By using this function we can directly evaluate the values of $\varepsilon_1 \circ \Phi^{-1}(\lambda)$ 
from the Frobenius representation \eqref{eq:2024dec5_1} without 
mapping $\lambda$ to the bit sequence $\Phi^{-1}(\lambda)$.


By Propositions \ref{prop:10} and \ref{prop:11},
we find that polynomial $Z_{2n+a,n}^{(r)}(q)$ in \eqref{eq:2024nov22_1}
can be viewed as a generating function enumerating
some subset of $\mathscr{P}^{(n)}_{n+a}$.
More explicitly, 
by classifying the elements of $\mathscr{P}^{(n)}_{n+a}$ by their $\mathcal{E}_1$-values,
we have
\begin{align}
\mathscr{P}^{(n)}_{n+a} &= \coprod_{r = 0}^n \mathscr{P}^{(n;r)}_{n+a}, \nonumber \\
\mathscr{P}^{(n;r)}_{n+a}  &:= \{ \lambda \in \mathscr{P}^{(n)}_{n+a} \mid 
\mathcal{E}_1 (\lambda; n) = r \}, \label{eq:2024dec6_2}
\end{align}
and by which we obtain the expression
\begin{equation}\label{eq:2024dec6_5}
Z_{2n+a,n}^{(r)}(q)=
\sum_{\eta \in \mathcal{H}^{(r)} (2n+a,n)} q^{E(\eta)}
 = \sum_{\lambda \in \mathscr{P}^{(n;r)}_{n+a} }
q^{|\lambda|}.
\end{equation}

\begin{corollary}
\par\noindent
\begin{enumerate}
\item $Z_{2n+a,n}^{(r)}(q) $ in \eqref{eq:2024nov22_1} enumerates 
the subset of $\mathscr{P}^{(n)}_{n+a}$ specified by $\mathcal{E}_1 (\bullet; n) = r$.
\item $B_n(q)$ in \eqref{eq:2024july2_1} enumerates 
the subset of $\mathscr{P}^{(n)}_{n+1}$ specified by even $\mathcal{E}_1 (\bullet; n)$.
\end{enumerate}
\end{corollary}
This provides an answer to the problem posed by
Andrews and Newman \cite[p.~9, Question II (3)]{AN2020}.

\section{An energy preserving bijection between partitions and pairs of partitions}\label{sec6}
Let $\mathscr{P}^{(n)} (\subset \mathscr{P})$ denote the integer partitions having at most $n$ parts, and let
$\mathscr{P}_l (\subset \mathscr{P})$ denote the integer partitions with each part
being less than or equal to $l$.

\begin{proposition}\label{prop:12}
There is an energy preserving bijection between
 $\mathscr{P}^{(n)}$ and $\mathscr{P}^{(n)}_{n+a} \times 
(\mathscr{P}_{2n+a} \setminus \mathscr{P}_{n+a}) $.
\end{proposition}
\begin{proof}
We construct such a bijection explicitly.
First we define the map $\Psi^{(a)} : \mathscr{P}^{(n)} \rightarrow \mathscr{P}^{(n)}_{n+a} \times (\mathscr{P}_{2n+a} \setminus \mathscr{P}_{n+a})$
as follows. Let $\nu \in \mathscr{P}^{(n)}$:
\begin{enumerate}
\item 
If $\nu$ has less than $n+a+1$ columns, then let $(\lambda, \mu) = (\nu, \emptyset)$.
Otherwise, let $\nu^{(0)} = \nu$.
\item
For $i \geq 1$, we define the sequence of partitions $\nu^{(i)}$ and the sequence of
rim hooks $\xi^{(i)}$ recursively by the following way.
If $\nu^{(i-1)}$ has more than $n+a$ columns, let $\xi^{(i)}$ be the longest of the
rightmost rim hooks
with arm length $n+a$ lying along the rim of $\nu^{(i-1)}$, and
let $\nu^{(i)} =\nu^{(i-1)} \setminus \xi^{(i)}$.
Repeat this procedure for $i=1,2,\dots$ until we have some $\nu^{(N)}$ that 
has less than $n+a+1$ columns.
Let $\lambda = \nu^{(N)}$ and $\mu = (|\xi^{(N)}|, \dots, |\xi^{(1)}|)$.
\item
Let $\Psi^{(a)} (\nu) = (\lambda, \mu)$.
\end{enumerate}
By definition we have $|\nu| = |\lambda| + |\mu|$, so $\Psi^{(a)}$ is an energy preserving map.
Note that, if we write $|\xi^{(i)}|=n+a+k_i$,
then the leg length of $\xi^{(i)}$ is $k_i-1$ and
we have $1 \leq k_1 \leq \dots \leq k_N \leq n$.
Thus we see that the $(\lambda, \mu)$ obtained here is
an element of $\mathscr{P}^{(n)}_{n+a} \times 
(\mathscr{P}_{2n+a} \setminus \mathscr{P}_{n+a})$.

Next we define the map $\Xi^{(a)} : \mathscr{P}^{(n)}_{n+a} \times (\mathscr{P}_{2n+a} \setminus \mathscr{P}_{n+a}) \rightarrow \mathscr{P}^{(n)} $
as follows. 
Let $(\lambda, \mu) \in \mathscr{P}^{(n)}_{n+a} \times 
(\mathscr{P}_{2n+a} \setminus \mathscr{P}_{n+a})$.
If $\mu = (\mu_1, \dots, \mu_N)$, then we can write $\mu_i = n+a+l_i$ where
the condition $n \geq l_1 \geq \dots \geq l_N \geq 1$ is satisfied. 
In the following, we interpret the coordinates of cells in such a way that
the coordinate $(1,1)$ is always at the position of the upper-leftmost cell of $\lambda$.

\begin{enumerate}
\item If $\mu = \emptyset$, then let $\nu = \lambda$.
Otherwise, we let $\lambda^{(0)} = \lambda$.
\item For $i=1, \dots, N$, we construct a partition $\lambda^{(i)}$ from the partition $\lambda^{(i-1)}$
by sticking a rim hook with length $\mu_i$ and leg length $l_i-1$
(which we call $\eta^{(i)}$) to it.
It is done in such a way that
the lower-leftmost cell of $\eta^{(i)}$ is at the coordinate $(l_i, \lambda^{(i-1)}_{l_i}+1)$, where
we interpret $\lambda^{(i-1)}_{l_i} =0$ if the partition $\lambda^{(i-1)}$ has
less than $l_i$ parts \footnote{This can occur only when $i=1$.}. Let $\nu = \lambda^{(N)}$.
\item Let $\Xi^{(a)} (\lambda, \mu) = \nu$.
\end{enumerate}
Since the leg length of rim hook $\eta^{(i)}$ is $l_i -1$, its arm length is always $n+a$.
This enables us to verify (by induction on $i$) that the upper-rightmost cell of $\eta^{(i)}$ is always in the first row
and to the right of the upper-rightmost cell of $\lambda^{(i-1)}$,
because  by the condition $l_{i-1} \geq l_i$ the lower-leftmost cell of $\eta^{(i)}$ must be always
to the right of the lower-leftmost cell of $\eta^{(i-1)}$.
Thus we see that maps $\Psi^{(a)}$ and $\Xi^{(a)}$ are mutually inverse, and
the $\nu$ obtained in this procedure is indeed an element of $\mathscr{P}^{(n)}$.
Therefore $\Psi^{(a)} = (\Xi^{(a)})^{-1}$ is
an energy preserving bijection from
 $\mathscr{P}^{(n)}$ to $\mathscr{P}^{(n)}_{n+a} \times 
(\mathscr{P}_{2n+a} \setminus \mathscr{P}_{n+a}) $.
\end{proof}

\begin{remark}
For every $\mathscr{S} \subset \mathscr{P} \sqcup (\mathscr{P} \times \mathscr{P})$ and $M \in \mathbb{Z}_{\geq 0}$, 
let $[\mathscr{S}]_M$ denote the subset of $\mathscr{S}$ that consists of
the elements with energy $M$.
As a result of Proposition \ref{prop:12}, 
the map $\Psi^{(a)}$ restricted to $[\mathscr{P}^{(n)} ]_M$ gives a bijection between
$[\mathscr{P}^{(n)} ]_M$ and $[\mathscr{P}^{(n)}_{n+a} \times 
(\mathscr{P}_{2n+a} \setminus \mathscr{P}_{n+a}) ]_M$, and hence
such pairs of subsets share a common cardinality for all $M$.
This fact is directly verified by using their generating functions as
\begin{align*}
&\sum_{M \geq 0} \# [\mathscr{P}^{(n)}_{n+a} \times 
(\mathscr{P}_{2n+a} \setminus \mathscr{P}_{n+a})]_M  q^M \\
&=
\sum_{M_1 \geq 0} \# [\mathscr{P}^{(n)}_{n+a} ]_{M_1}  q^{M_1}
 \times
\sum_{M_2 \geq 0} \# [(\mathscr{P}_{2n+a} \setminus \mathscr{P}_{n+a})  ]_{M_2}   q^{M_2}
\\
&={ 2n+a \brack n} \cdot \frac{(q;q)_{n+a}}{(q;q)_{2n+a}} =\frac{1}{(q;q)_n} \\
&= \sum_{M \geq 0} \# [\mathscr{P}^{(n)} ]_M   q^M.
\end{align*}
\end{remark}

\begin{example}
Let $n=2, a=1$, and $M=10$.
Then the bijection $\Psi^{(1)} : [\mathscr{P}^{(2)} ]_{10} \rightarrow [\mathscr{P}^{(2)}_{3} \times 
(\mathscr{P}_{5} \setminus \mathscr{P}_{3}) ]_{10}$ is explicitly given by
\begin{align*}
\Psi^{(1)} ((10)) &= ((2),(4,4)), & \Psi^{(1)} ((9,1)) &= ((1,1),(4,4)), \\
\Psi^{(1)} ((8,2)) &= ((1),(5,4)), & \Psi^{(1)} ((7,3)) &= ((3,3),(4)), \\
\Psi^{(1)} ((6,4)) &= ((3,2),(5)), & \Psi^{(1)} ((5,5)) &= (\emptyset,(5,5)).
\end{align*}
See Fig.~\ref{fig6} for the corresponding combinatorial procedure.
\begin{figure}[h]
\centering
\includegraphics[height=3cm]{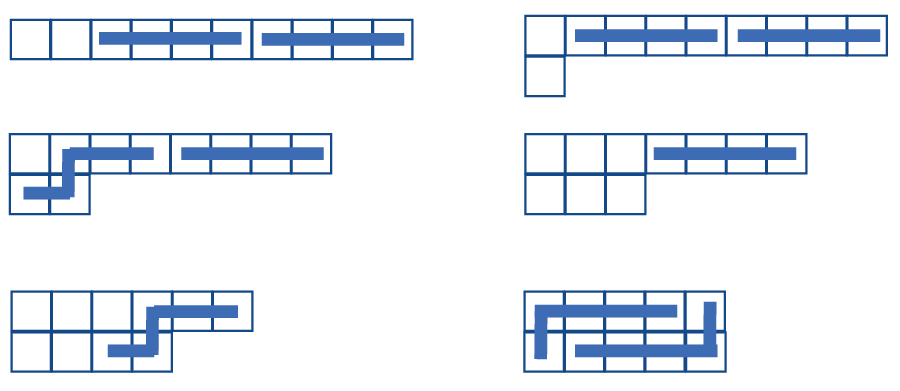}
\caption{A graphical representation for
the combinatorial procedure to realize the bijection $\Psi^{(1)} : [\mathscr{P}^{(2)} ]_{10} \rightarrow [\mathscr{P}^{(2)}_{3} \times 
(\mathscr{P}_{5} \setminus \mathscr{P}_{3}) ]_{10}$.}
\label{fig6}
\end{figure}

\end{example}

Having $\Psi^{(a)}$ in the proof of Proposition \ref{prop:12},
we let $\Psi^{(a)}_1$ and $\Psi^{(a)}_2$ be the pair of maps 
obtained from $\Psi^{(a)}$ by composing with a projection
into the first and the second component of $\mathscr{P}^{(n)}_{n+a} \times 
(\mathscr{P}_{2n+a} \setminus \mathscr{P}_{n+a})$,
respectively.

\begin{corollary}\label{coro:13}
There is an energy preserving bijection between
 $\mathscr{P}^{(n)} \times \mathscr{P}_{n+a}$ and $\mathscr{P}^{(n)}_{n+a} \times 
\mathscr{P}_{2n+a} $.
\end{corollary}
\begin{proof}
We define the map $\Phi^{(a)}: \mathscr{P}^{(n)} \times \mathscr{P}_{n+a} \rightarrow
\mathscr{P}^{(n)}_{n+a} \times 
\mathscr{P}_{2n+a}$ as follows.
Given an element $(\nu, \pi) \in \mathscr{P}^{(n)} \times \mathscr{P}_{n+a}$,
let $\Phi^{(a)}(\nu,\pi) = (\Psi^{(a)}_1(\nu),  \Psi^{(a)}_2(\nu) \cup \pi )$
\footnote{For every pair of partitions $(\lambda, \mu)$, we define $\lambda \cup \mu$ to be
the partition whose parts are those of $\lambda$ and $\mu$, arranged in decrease order \cite{Macdonald}. In fact, this arrangement is unnecessary in the above case.}.
Clearly this $\Phi^{(a)}$ gives the desired bijection.
\end{proof}

\section{Proof of the main result}\label{sec7}
In this section, we give a proof of Theorem \ref{th:main} to establish the main
result of this paper.

Let $\lambda \in \mathscr{P}$, and recall the definitions of 
the Durfee square/rectangle $D_a(\lambda)$, 
its vertical side $n_a =n_a(\lambda)$,
the sub-diagram $A_a(\lambda)$,
and the residual diagram $R_a(\lambda)$ used to
evaluate its sqrank/rerank.
In $\lambda$,
the cells inside the first $n_a + a$ columns but not in $D_a(\lambda)$ constitute a sub-diagram 
of $\lambda$. Let $L_a(\lambda)$ denote this sub-diagram.
In other words, we have a unique decomposition
\begin{equation}\label{eq:2024dec6_1}
\lambda = D_a(\lambda) \uplus A_a(\lambda) \uplus L_a(\lambda),
\end{equation}
for every partition $\lambda$, where $\uplus$ denotes disjoint union \red{of sets}.
Note that $A_a(\lambda) \in \mathscr{P}^{(n)}$ and
$L_a(\lambda) \in \mathscr{P}_{n+a}$
with $n = n_a(\lambda)$.

\begin{example}
Let $\lambda = (19,16,9,2,1)$.
Then we have 
$L_0(\lambda)=L_1(\lambda)=(2,1)$.
See Figs \ref{fig1} (a) and \ref{fig2} (a) for the relevant decomposition of the diagram into
the sub-diagrams in \eqref{eq:2024dec6_1}.
\end{example}

By comparing the procedure to define sqrank/rerank and
the procedure to define $\Psi^{(a)}$ in the proof of Proposition \ref{prop:12},
we see that $R_a(\lambda) = \Psi^{(a)}_1(A_a(\lambda)) \in \mathscr{P}^{(n)}_{n+a}$ with $n = n_a(\lambda)$.
Thus in terms of the map $\Psi^{(a)}_1$ the definition of sqrank/rerank takes the form
\begin{align}
{\rm sqrank} (\lambda) &= \mathcal{E}_1 ( \Psi^{(0)}_1(A_0(\lambda)); n_0(\lambda)), \\
{\rm rerank} (\lambda) &= \mathcal{E}_1 ( \Psi^{(1)}_1(A_1(\lambda)); n_1(\lambda)), 
\end{align}
where $\mathcal{E}_1$ is the function defined in \eqref{eq:2024nov27_6}.
Now we define
\begin{align}
\mathscr{S}^{[n;r]} &= \{ \lambda \in \mathscr{P} \mid {\rm sqrank} (\lambda) =r, 
n_0(\lambda) =n \}, \\
\mathscr{R}^{[n;r]} &= \{ \lambda \in \mathscr{P} \mid {\rm rerank} (\lambda) =r, 
n_1(\lambda) =n \}.
\end{align}

Recall the definition of
$\Phi^{(a)}: \mathscr{P}^{(n)} \times \mathscr{P}_{n+a} \rightarrow
\mathscr{P}^{(n)}_{n+a} \times \mathscr{P}_{2n+a}$ 
in the proof of Corollary \ref{coro:13} and the decomposition \eqref{eq:2024dec6_1}.
By using the relation
\begin{equation}
\Phi^{(a)} (A_a(\lambda), L_a(\lambda)) = (\Psi^{(a)}_1(A_a(\lambda)),\Psi_2^{(a)} (A_a(\lambda)) \cup  L_a(\lambda)),
\end{equation}
and noting that the map $\Phi^{(a)}$ is an energy preserving bijection,
we see that
\begin{align}
\sum_{\lambda \in \mathscr{S}^{[n;r]} }q^{|\lambda|} &= 
q^{n^2} \cdot 
\sum_{\pi \in \mathscr{P}^{(n;r)}_{n} }q^{|\pi|} \cdot
\sum_{\mu \in \mathscr{P}_{2n} }q^{|\mu|}, \label{eq:2024dec6_3}\\
\sum_{\lambda \in \mathscr{R}^{[n;r]} }q^{|\lambda|} &= 
q^{n^2+n} \cdot 
\sum_{\pi \in \mathscr{P}^{(n;r)}_{n+1} }q^{|\pi|} \cdot
\sum_{\mu \in \mathscr{P}_{2n+1} }q^{|\mu|},\label{eq:2024dec6_4}
\end{align}
where $\mathscr{P}^{(n;r)}_{n+a}$ is the subset of $\mathscr{P}^{(n)}_{n+a}$ given by
\eqref{eq:2024dec6_2}.
Now we define
\begin{align}
{\rm ps}^{(r)} (n) &= \# \{ \lambda \vdash n \mid {\rm sqrank} (\lambda) =r \}, \\
{\rm pr}^{(r)} (n) &= \# \{ \lambda \vdash n  \mid {\rm rerank} (\lambda) =r \}.
\end{align}
Then by using \eqref{eq:2024dec6_5}, \eqref{eq:2024dec6_3}, and \eqref{eq:2024dec6_4},
we have 
\begin{align*}
\sum_{n \geq 0} {\rm ps}^{(r)} (n) q^n
&=  \sum_{n \geq 0} \frac{q^{n^2} Z_{2n,n}^{(r)}(q)}{(q;q)_{2n}},  \\
\sum_{n \geq 0} {\rm pr}^{(r)} (n) q^n 
&=  \sum_{n \geq 0} \frac{q^{n^2+n} Z_{2n+1,n}^{(r)}(q)}{(q;q)_{2n+1}}.
\end{align*}
Comparing this result with
Proposition \ref{prop:2024oct16_2}, we see that $p^{(r)}_{2,1}(n)={\rm ps}^{(r)} (n)$ 
and $p^{(r)}_{2,2}(n)={\rm pr}^{(r)} (n)$.
This completes the proof of Theorem \ref{th:main}.

\section{Concluding remarks}\label{sec8}
Inspired by the study of the minimal excludant in integer partitions by
Andrews and Newman \cite{AN2020}, and in particular by the
generating function \eqref{eq:2024dec25_1} of
the mex-related function $p_{2,2}(n)$, we introduced a pair of partition statistics, sqrank and rerank.
We have shown that there is a nontrivial connection between these statistics and
the odd/even mex functions for having equinumerous integer partitions.
To establish our main result, Theorem \ref{th:main}, we used an analytic method for
showing that the different partition statistics
share the same generating functions.
We hope that there is a more explicit bijective proof of our result.  

An obvious generalization of our answer (Corollary \ref{coro:2}) 
to one of their questions \cite[p.~9, Question II]{AN2020}
may be stated as follows: {\em $p_{2,1}(n)$ equals the number of partitions of $n$ with even sqrank}.
However, as was shown in their paper \cite[Theorem 4]{AN2020}, 
there is a simpler partition statistic which produces the same integer sequence $\{p_{2,1}(n) \}_{n \geq 0}$.
In fact, there is the identity $p_{2,1}(n) = p_e(n)$ where $p_e(n)$ is the number of partitions of $n$
into an even numbers of parts.
The author does not know whether a similar result exists for $p_{2,2}(n)$.

The author believes that most of the readers will agree to saying that 
the definition of our new partition statistics in Sect.~\ref{sec2} is fairly succinct.
However, while the definition of rank and crank was stated in three lines
without using mathematical notation \cite{AN2020}, 
our definition of sqrank and rerank requires ten lines. 
While the former only uses parts and number of parts of the partitions,
the latter uses Ferrers diagrams and combinatorial procedures thereon. 
This may be rather unsatisfactory.
We hope that someone will be able to find another partition statistic with more succinct
definition but can also reproduce the same integer sequence of $p_{2,2}(n)$.

\backmatter


\section*{Declarations}

\begin{itemize}
\item Conflict of interest: The author confirms that he has no conflict of interest
in connection with this paper.
\item Data availability: No data were generated or used in the preparation of this
paper.
\end{itemize}


\begin{appendices}

\section{Relation to the Kostka \red{polynomials}}\label{secA}
The following polynomial considered in Sect. \ref{sec4_2}
\begin{equation}\label{eq:2025jun16_3}
Z_{L,s}^{(0)}(q) :=
\sum_{\eta \in \mathcal{H}^{(0)} (L,s)} q^{E(\eta)} ={ L \brack s}- { L \brack s-1},
\end{equation}
is equal to the Kostka \red{polynomial} $K_{\lambda, \mu}(q)$ with
$\lambda =2^s 1^{L-2s}$ and $\mu =1^L$.
We present two explanations for this fact.

The first is to connect its right hand side with this $K_{\lambda, \mu}(q)$,
pointed out by Warnaar \cite{Warnaar2004}.
As was suggested there, it can be proved by using a formula
derived by Macdonald \cite[III. \S 6. Exercise 2]{Macdonald}
\begin{equation*}
K_{\lambda, (1^L)}(q) = q^{n(\lambda')} (q;q)_L H_{\lambda} (q)^{-1}, \quad
n(\lambda') = \sum_{i \geq 1} {\lambda_i \choose 2},
\end{equation*}
for any partition $\lambda$ of $L$, where $H_{\lambda} (q)$ is the hook-length polynomial
associated with $\lambda$ defined by
\begin{equation*}
H_{\lambda} (q) = \prod_{x \in \lambda} (1-q^{h(x)}).
\end{equation*}
Here, $h(x)$ is the length of the hook whose body is on cell $x$ of the partition 
$\lambda$ (viewed as a Ferrers diagram).
For $\lambda =2^s 1^{L-2s}$, we have $n(\lambda') =s$ and
\begin{equation*}
H_{\lambda} (q) =(q;q)_s (q;q)_{L-2s} (q^{L-2s+2};q)_s = \frac{(q;q)_s (q;q)_{L-s+1}}{1-q^{L-2s+1}}.
\end{equation*}
Therefore
\begin{equation*}
K_{(2^s 1^{L-2s}), (1^L)}(q) = \frac{(q;q)_L}{(q;q)_s (q;q)_{L-s+1}}(1-q^{L-s+1}-1+q^s)
={ L \brack s}- { L \brack s-1}.
\end{equation*}

The second is to connect its left hand side with the following
Lascoux-Sch\"utzenberger's formula
\begin{equation*}
K_{\lambda, (1^L)}(q) = \sum_{T} q^{c(T)}
\end{equation*}
summed over all standard tableaux $T$ of shape $\lambda$,
with an interpretation of the charge $c(T)$ by Nakayashiki and Yamada \cite{NY97}.
For $\lambda =2^s 1^{L-2s}$, there is a bijection between $\mathcal{H}^{(0)} (L,s)$ and
the standard tableaux of this shape.
That is, for every bit sequence
$\eta =(\eta_1,\dots, \eta_L) \in \mathcal{H}^{(0)} (L,s)$ we send it to a 
unique standard tableau $T(\eta)$ of shape $2^s 1^{L-2s}$ in which
letter $j$ appears in the first (resp. the second) column of $\lambda$ if
$\eta_j =0$ (resp. $\eta_j =1$).
From their main theorem \cite[Theorem 4.2]{NY97}, we have an expression for the index ${\rm ind}(i)$
of letter $i$ in the 
standard tableau $T(\eta)$ as a sum of  ``winding numbers"  \cite[Example 2.15]{NY97} as
\begin{equation*}
{\rm ind}(i) = \sum_{j=1}^{i-1} H(\eta_j,\eta_{j+1}),
\end{equation*}
where $H: \{0,1\}^2 \rightarrow \{0,1\}$ is the function defined
above equation \eqref{eq:2024july3_9}. 
Then the charge $c(T(\eta))$ is given by
\begin{equation*}
c(T(\eta)) = \sum_{i=1}^L {\rm ind}(i)  = \sum_{1 \leq j \leq L-1} (L-j) H(\eta_j,\eta_{j+1}).
\end{equation*}
Now we consider a map $\sigma$ from $\mathcal{H}^{(0)} (L,s)$ to itself as follows.
For every $\eta \in \mathcal{H}^{(0)} (L,s)$, let $\sigma (\eta)$ be the bit sequence
obtained from $\eta$ by first reversing its order, and then applying on it the time evolution
of box-ball system (See, for example \cite{CKST2023, Takagi1}) once:
\begin{align*}
&0000111 \xrightarrow{\text{reverse}} 1110000 \xrightarrow{\text{time evolve}} 0001110=\sigma(0000111) ,\\
&0100110 \xrightarrow{\text{reverse}} 0110010 \xrightarrow{\text{time evolve}} 0001101=\sigma(0100110),\\
&0010110 \xrightarrow{\text{reverse}} 0110100 \xrightarrow{\text{time evolve}} 0001011=\sigma(0010110).
\end{align*}
It is easy to see that $\sigma$ is an involution on the whole $\mathcal{H}^{(0)} (L,s)$, and 
$c(T(\eta)) = E(\sigma (\eta))$.
Therefore
\begin{equation*}
K_{(2^s 1^{L-2s}), (1^L)}(q) = \sum_{\eta \in \mathcal{H}^{(0)} (L,s)} q^{c(T(\eta))} 
=\sum_{\eta \in \mathcal{H}^{(0)} (L,s)} q^{E(\sigma (\eta))} 
=\sum_{\eta \in \mathcal{H}^{(0)} (L,s)} q^{E(\eta)}. 
\end{equation*}

\section{Relation to the affine Lie algebra character formulae}\label{secB}
The second description of the Kostka polynomials in Appendix \ref{secA} enables us to define
a generalization of this polynomial to its level-restricted analogue.
By classifying the elements of $\mathcal{H}^{(0)} (L,s)$ by the maxima of their path encodings, 
we have
\begin{align*}
\mathcal{H}^{(0)} (L,s) &= \coprod_{l= 1}^{L-s} \mathcal{H}^{(0,l)} (L,s),\\
\mathcal{H}^{(0,l)} (L,s) &:= \{ \eta \in \mathcal{H}^{(0)} (L,s) \mid
\max_{1 \leq i \leq L} S_i(\eta) = l \}. 
\end{align*}
Then the {\em level restricted} Kostka polynomials can be defined as
\begin{equation*}
K_{(2^s 1^{L-2s}), (1^L)}^{(l)}(q) = \sum_{\eta \in \mathcal{H}^{(0,l)} (L,s)} q^{c(T(\eta))}
=\sum_{\eta \in \mathcal{H}^{(0,l)} (L,s)} q^{E(\sigma(\eta))}. 
\end{equation*}
Note that the map $\sigma$ defined in Appendix \ref{secA} is not an involution on each $\mathcal{H}^{(0,l)} (L,s)$.
Now we see that the identities in Proposition \ref{prop:2024oct16_2}
can be written as
\begin{align}
\sum_{n \geq 0} p^{(r)}_{2,1}(n) q^n
&=  \sum_{n \geq 0} \frac{K_{(2^{n-r} 1^{2r}), (1^{2n})}(q) K_{(2^n), (1^{2n})}^{(1)}(q)}{(q;q)_{2n}}, \label{eq:2025jun20_1} \\
\sum_{n \geq 0} p^{(r)}_{2,2}(n) q^n 
&=  \sum_{n \geq 0} \frac{K_{(2^{n-r} 1^{2r+1}), (1^{2n+1})}(q) K_{(2^n1), (1^{2n+1})}^{(1)}(q)}{(q;q)_{2n+1}}.\label{eq:2025jun20_2}
\end{align}
This can be verified by noticing that both
$\mathcal{H}^{(0,1)} (2n,n) = \{ 01 \dots 01 \}$ and
$\mathcal{H}^{(0,1)} (2n+1,n) = \{ 01 \dots 010 \}$
have only one element, and hence the relevant level restricted Kostka polynomials are actually
monomials with the powers
$E(\sigma(01 \dots 01)) = E(01 \dots 01) = 1+3+\dots + (2n-1) = n^2$ and 
$E(\sigma(010 \dots 010)) = E(00101 \dots 01) = 2+4+\dots + 2n = n^2+n$,
respectively.

The right hand sides of the above identities are special cases of
the more general spinon character formulae proposed by Hatayama et.~al. \cite[(5.49)]{HKKOTY},
for the branching coefficients for level 1 highest weight modules of
affine Lie algebra \red{$\widehat{\mathfrak{sl}}_2$} with respect to its irreducible decomposition into
finite dimensional modules of its subalgebra \red{$\mathfrak{sl}_2$}.
Following their notation \cite{HKKOTY}, 
the right hand side of \eqref{eq:2025jun20_1} is equal to
$b^{V(\Lambda_0)}_{2r \overline{\Lambda}_1}(q)$, and that of
\eqref{eq:2025jun20_2} is
$b^{V(\Lambda_1)}_{(2r+1) \overline{\Lambda}_1} (q)$.
Here, the superscript $V(\Lambda_i) \, (i=0,1)$ indicates that it is for
the irreducible highest weight module of the \red{$\widehat{\mathfrak{sl}}_2$} with highest weight $\Lambda_i$, and
the subscript $m \overline{\Lambda}_1 \, (m=2r, 2r+1)$ indicates
that it is for the $m+1$ dimensional irreducible module of the \red{$\mathfrak{sl}_2$}.

This provides us another interpretation of the mex-related functions.
In each degree $n$ measuring the depth along the direction of ``null root"
in the weight space of the \red{$\widehat{\mathfrak{sl}}_2$},
the number of $2r+1$ dimensional \red{$\mathfrak{sl}_2$} modules in the irreducible decomposition of
$V(\Lambda_0)$ is
equal to $p^{(r)}_{2,1}(n)$, and that of $2r+2$ dimensional \red{$\mathfrak{sl}_2$} modules in the irreducible decomposition of $V(\Lambda_1)$ is equal to $p^{(r)}_{2,2}(n)$.

A direct connection of the left hand sides of \eqref{eq:2025jun20_1} and
\eqref{eq:2025jun20_2} with the above branching coefficients can be seen as follows.
We follow the notation by Warnaar \cite{Warnaar2001}.
He considered more general highest weights $(p'/p - \ell - 2) \Lambda_0 + \ell \Lambda_1$
and denoted relevant characters by $\chi_\ell^{(p,p')}(z,q)$.
Here we restrict ourselves to only level 1 cases ($p=1, p'=3$).
Then the character formula can be written as
\begin{equation*}
\chi_\ell^{(1,3)}(z,q) = q^{\frac18 -\frac{(\ell + 1)^2}{12}} \sum_{m \in \mathbb{Z}}
C^{(1,3)}_{m,\ell}(q) z^{-\frac12 m},
\end{equation*}
with the string function $C^{(1,3)}_{m,\ell}(q)$ given by
\begin{equation*}
C^{(1,3)}_{m,\ell}(q) = \frac{q^{\frac14 (m^2 - \ell^2)}}{(q;q)_\infty},
\end{equation*}
if $m+\ell \equiv 0 \pmod{2}$, and $0$ otherwise.
It is easy to see that
\begin{align*}
\chi_0^{(1,3)}(z,q) &= q^{\frac{1}{24}} \sum_{r=0}^\infty (C^{(1,3)}_{2r,0}(q) - C^{(1,3)}_{2r+2,0}(q))
{\rm ch}_{r} (z),\\
\chi_1^{(1,3)}(z,q) &= q^{-\frac{5}{24}} \sum_{r=0}^\infty (C^{(1,3)}_{2r+1,1}(q) - C^{(1,3)}_{2r+3,1}(q))
{\rm ch}_{r+\frac12} (z),
\end{align*}
where ${\rm ch}_{j} (z) :=(z^{j+1/2} - z^{-j-1/2})/(z^{1/2} - z^{-1/2})$
for $j \in \frac12 \mathbb{Z}_{\geq 0}$ is the character of the $2j+1$ dimensional irreducible 
\red{$\mathfrak{sl}_2$} module.
Therefore we have
\begin{align*}
b^{V(\Lambda_0)}_{2r \overline{\Lambda}_1}(q) &= C^{(1,3)}_{2r,0}(q) - C^{(1,3)}_{2r+2,0}(q) =
\frac{q^{r^2}(1 - q^{2r+1})}{(q;q)_{\infty}},\\
b^{V(\Lambda_1)}_{(2r+1) \overline{\Lambda}_1} (q) &=
C^{(1,3)}_{2r+1,1}(q) - C^{(1,3)}_{2r+3,1}(q)= 
\frac{q^{r(r+1)}(1 - q^{2(r+1)})}{(q;q)_{\infty}}.
\end{align*}
As we explained in the proof of Proposition \ref{prop:2024oct16_2}, the right hand sides
of these equations are the generating functions of $p^{(r)}_{2,1}(n)$ and $p^{(r)}_{2,2}(n)$,
respectively.

\section{Relation to the configuration sums by Andrews-Baxter-Forrester}\label{secC}
Andrews, Baxter and Forrester \cite[Lemma 2.6.1]{ABF}  obtained the following
expression for a one-dimensional configuration sum
(in the case of ``regime II" in their terminology)
\begin{align*}
x_m(a,b,c) &= q^{(m+a-b)/4} \left\{
f_m (a,b,c) - f_m(-a,b,c) \right\}, \\
f_m(a,b,c)&=\sum_{\lambda = - \infty}^\infty 
q^{r \lambda^2 - a \lambda + (b+1-c)(2r\lambda + b-a)/4}
{ m \brack \frac12 (m+a-b) - r \lambda }.
\end{align*}
Here $r$ is a fixed integer $(\geq 3)$ and $m,a,b,c$ are integers satisfying 
$m \geq 0, 1 \leq a,b,c \leq r, c=b\pm 1$ and $m+a-b \equiv 0 \pmod{2}$.
By taking the limit $r \rightarrow \infty$, only the $\lambda =0$ term survives in the sum over $\lambda$
and one has
\begin{equation*}
x_m(a,b,c) = q^{(m+a-b)/4} \left\{
q^{\frac{(b+1-c)(b-a)}{4}} { m \brack \frac12 (m+a-b) } -
q^{\frac{(b+1-c)(b+a)}{4}} { m \brack \frac12 (m-a-b) }\right\}.
\end{equation*}
In particular for $m=L, a=1,b=L+1-2s$ and $c=b\pm1$, we have
\begin{align}
x_L(1,L+1-2s, L+2-2s) =
q^{\frac{s}{2}} \left\{
{ L \brack s } -
 { L \brack s-1 }\right\}, \label{eq:2025jun16_1}\\
x_L(1,L+1-2s, L-2s) =
q^{\frac{L-s}{2}} \left\{
{ L \brack s } -
q { L \brack s-1 }\right\}.\label{eq:2025jun16_2}
\end{align}

Let $H_{\rm ABF}$ denote the function from $ \{0,1\}^2$ to $\{0,1\}$ that is given by
$H_{\rm ABF}(0,1)=H_{\rm ABF}(1,0)=1/2, H_{\rm ABF}(0,0)=H_{\rm ABF}(1,1)=0$.
For bit sequences $\eta \in \mathcal{H}^{(0)} (L,s)$, we define their ABF-energy by
\begin{equation*}
E_{\rm ABF}(\eta) = \sum_{1 \leq j \leq L-1} j H_{\rm ABF}(\eta)(\eta_j, \eta_{j+1}).
\end{equation*}
Also we define its relevant configuration sums
\begin{align*}
Z_{L,s}^{\nearrow} (q)&= \sum_{\eta \in \mathcal{H}^{(0)} (L,s)} q^{E_{\rm ABF}(\eta) + L H_{\rm ABF}(\eta_L,0)} ,\\
Z_{L,s}^{\searrow} (q)&=\sum_{\eta \in \mathcal{H}^{(0)} (L,s)} q^{E_{\rm ABF}(\eta) + L H_{\rm ABF}(\eta_L,1)}.
\end{align*}
They satisfy the boundary conditions $Z_{L,0}^{\nearrow} (q) =1, Z_{L,0}^{\searrow} (q) = q^{L/2},
Z_{2s-1,s}^{\nearrow} (q) =0$ and the recursion relations \cite[(2.6.2) and (2.6.3)]{ABF}
\begin{align*}
Z_{L,s}^{\nearrow} (q)&= q^{L/2}Z_{L-1,s-1}^{\searrow} (q) + Z_{L-1,s}^{\nearrow} (q),\\
Z_{L,s}^{\searrow} (q)&= Z_{L-1,s-1}^{\searrow} (q) + q^{L/2}Z_{L-1,s}^{\nearrow} (q).
\end{align*}
By using \eqref{eq:2024july3_1} and \eqref{eq:2024july3_2}, we see that
the right hand sides of
\eqref{eq:2025jun16_1} and \eqref{eq:2025jun16_2} satisfy
the same boundary conditions and recursion relations.
This implies that
\begin{align}
Z_{L,s}^{\nearrow} (q)&= q^{\frac{s}{2}} \left\{
{ L \brack s } -
 { L \brack s-1 }\right\}, \nonumber\\
Z_{L,s}^{\searrow} (q)&= q^{\frac{L-s}{2}} \left\{
{ L \brack s } -
q { L \brack s-1 }\right\}.\label{eq:2025jun16_4}
\end{align}
The argument presented here is essentially equivalent to the one for deriving
their results \cite[Lemma 2.6.1]{ABF},
restricted to a limiting case and rewritten in our notation.

The $r=0$ case of our one dimensional configuration sum \eqref{eq:2024oct18_1}
defined by using the energy \eqref{eq:2024july3_9} takes the form
\eqref{eq:2025jun16_3}, which differs from the first equation of \eqref{eq:2025jun16_4} only by 
the factor $q^{\frac{s}{2}}$.
This factor can be obtained in the following way.
For $\eta =(\eta_1, \dots, \eta_L) \in \mathcal{H}^{(0)} (L,s)$, let
$\eta^{\nearrow} =(\eta_1, \dots, \eta_L, 0) \in \mathcal{H}^{(0)} (L+1,s)$.
Since $\eta_1 =0$, the path $S(\eta^{\nearrow})$ has the same numbers of
local maxima and local minima \footnote{See Fig.~\ref{fig:pathenc2} for an example.}.
Let $\alpha_1, \dots, \alpha_p$ and $\beta_1, \dots, \beta_p$ be the positions
at which the path $S(\eta^{\nearrow})$ takes its local maxima and local minima, respectively,
under the condition $\alpha_1 < \beta_1 < \dots < \alpha_p < \beta_p$.
Then we have
\begin{align*}
E_{\rm ABF}(\eta) + L H_{\rm ABF}(\eta_L,0) &= \frac12 (\alpha_1 + \beta_1+\dots+\alpha_p + \beta_p),\\
E(\eta) + L H(\eta_L,0) &=\alpha_1 +\dots+\alpha_p.
\end{align*} 
The latter holds because $H(\eta_L,0) =0$ for any $\eta_L$.
Therefore
\begin{equation*}
E_{\rm ABF}(\eta) + L H_{\rm ABF}(\eta_L,0) - (E(\eta) + L H(\eta_L,0))=
\frac12 \left( (\beta_1 - \alpha_1) + \dots + (\beta_p - \alpha_p) \right).
\end{equation*}
By noticing the definition of the path encoding \eqref{eq:2024nov27_4},
one sees that $ (\beta_1 - \alpha_1) + \dots + (\beta_p - \alpha_p)=s$, the numbers of
$1$'s in the bit sequence $\eta$. 
\end{appendices}

\end{document}